\documentclass[UTF-8,fontset=windows,final]{siamltex}
\usepackage{tabularx}

\usepackage{subcaption}

\usepackage{multirow}
\usepackage{amsmath, amsfonts, amssymb, mathrsfs}
\usepackage{booktabs, graphics, graphicx, diagbox, enumerate, multirow}
\usepackage{algorithm, algorithmic, tikz}
\usepackage[noadjust]{cite}
\usepackage[top=3cm, bottom=3cm, left=2cm, right=2cm]{geometry}

\usetikzlibrary{arrows.meta}

\renewcommand{\div}{\operatorname{div}}

\renewcommand\abstractname{Abstract}

\renewcommand\figurename{Figure}
\renewcommand\tablename{Table}

\renewcommand\refname{References}

\numberwithin{equation}{section}

\newtheorem{problem}{Problem}

\allowdisplaybreaks[4]

\normalsize

\newtheorem{remark}{Remark}[section]

\usepackage{hyperref, cleveref}

\hypersetup{
hidelinks,
colorlinks=true,
linkcolor=blue,
citecolor=blue,
urlcolor = blue
}

\crefname{section}{Section}{Sections}
\crefname{subsection}{Section}{Sections}
\crefname{figure}{Figure}{Figures}
\crefname{table}{Table}{Tables}
\crefname{theorem}{Theorem}{Theorems}
\crefname{lemma}{Lemma}{Lemmas}
\crefname{corollary}{Corollary}{Corollaries}
\crefname{define}{Definition}{Definitions}
\crefname{problem}{Problem}{Problems}
\crefname{example}{Example}{Examples}
\crefname{testcase}{Testcase}{Testcases}

\begin{document}
	
	
	\title{A Parallel Solver with Multiphysics Finite Element Method for Poroelasticity Coupled with Elasticity Model\footnote{Last update: \today}}
	
	\author{
		Zhihao Ge\thanks{School of Mathematics and Statistics, Henan University, Kaifeng 475004, PR China ({\tt Email:zhihaoge@henu.edu.cn}).
			The work of this author was supported by National Natural Science Foundation of China(No. 12371393) and Natural Science Foundation of Henan Province(No. 242300421047).}
		\and
		Chengxin Wang\thanks{School of Mathematics and Statistics, Henan University, Kaifeng 475004, P.R. China.}
		%
	}
	
	\maketitle
	
	
	\setcounter{page}{1}
	
	
	
	\begin{abstract}
    In this paper, we propose a parallel solver for solving the quasi-static linear poroelasticity coupled with linear elasticity model in the Lagrange multiplier framework. Firstly, we reformulate the model into a coupling of the nearly incompressible elasticity and an unsteady affection-diffusion equations by setting new variable ``elastic pressure" and ``volumetric fluid content". And we introduce a Lagrange multiplier to guarantee the normal stress continuity on the interface. Then, we give the variational formulations in each subdomain and choose the $\boldsymbol{P}_k$-$P_1$-$P_1$ mixed finite element tuple for poroelasticity subdomain, and $\boldsymbol{P}_k$-$P_1$ finite element pair ($k=1,2$) for elasticity subdomain and the backward Euler scheme for time. Also, we  propose a parallel solver for solving the fully discrete scheme at each time step-- the FETI method with a classical FETI preconditioner for solving the Lagrange multiplier and calculating the subproblems in each subdomain in parallel. And we show several numerical tests to validate the computational efficiency and the convergence error order, and we consider Barry-Mercer's model as the benchmark test to show that there no oscillation in the computed pressure. Finally, we draw conclusions to summarize the main results of this paper.
    \end{abstract}
    
    
    
    \begin{keywords}
    Poroelasticity, elasticity, multiphysics finite element method, domain decomposition, Schur complement method, FETI method
    
    \end{keywords}
	
	\begin{AMS}
		65N30. 
	\end{AMS}
	
	\pagestyle{myheadings}
	\thispagestyle{plain}
	\markboth{ZHIHAO GE, CHENGXIN WANG}{ A PARALLEL SOLVER WITH MFEM FOR POROELASTICITY COUPLED WITH ELASTICITY MODEL}
	

\section{Introduction}
\label{sect:introduction}

In this paper, we study the non-overlapping coupled poroelasticity and elasticity model, which is a challenging multiphysics and multidomain problem with applications to multiple industrial sectors. One typical practical usage is the simulation to the reservoir\cite{abdi_modeling_2021, dai_co2_2017, feng_phase-field_2023}. For example, attaching increasing importance to energy conservation and carbon reduction, the carbon dioxide underground injection sequestration techniques, as an efficient selection for reducing the greenhouse gases in the atmosphere, is now a popular research topic, in which one of the motivation is that the supercritical CO$_2$ fracturing has better performance in fluid fracturing and oil production than hydraulic fracturing. In addition, the model of coupled poroelasticity and elasticity can be applied in the biomechanical engineering\cite{yankova_study_2020, chen_physical_2023}, such as simulation of the brain injury and tumor growth.

The coupled poroelasticity and elasticity model consists of at least two different subdomains, that is, the poroelastic "pay-zone" and elastic "nonpay-zone". We adopt the Biot's\cite{biot_general_1941, terzaghi_theoretical_1943} system of poroelasticity to model the poroelastic media in the pay-zone and a classic elasticity model in the nonpay-zone. The coupled model can be seen as an elastic model in which part of the domain has a porous structure and is saturated by a fluid\cite{coussy_poromechanics_2004}. For the respective models in each subdomain.

The coupled poroelasticity and elasticity model was first studied by Girault et al\cite{girault_domain_2011}, in which the displacement $\mathbf u$ was partitioned into two parts, one of which is dependent on the fluid pressure $p$ and the other is independent. And for the numerical study, the discontinuous Galerkin jumps and mortars method was adopted for semi-discretizing the model, and the corresponding error estimates was given. A Mandel's model was studied as a numerical example, in which the errors of the model with its order were validated.

In practical actuality, various discretization schemes of the coupled poroelasticity and elasticity model can be chosen. For the model of poroelasticity, Phillips and Wheeler\cite{phillips_coupling_2007, phillips_coupling_2007-1} proposed a mixed finite element method for solving a quasi-linear model of poroelasticity. However, the tendency of Lam\'e constant to infinity $\lambda\to\infty$ will lead to a so-called locking phenomenon, that is, the solution will lose convergence if $\lambda$ is set to be sufficiently large. To overcome the locking, one of the efficient way is to adopt the discontinuous Galerkin method\cite{phillips_overcoming_2009}. Another strategy is the multiphysics finite element method\cite{feng_analysis_2018}, where two new variables were introduced so that the elastic momentum balance part was rewritten in a generalized Stokes form, and validated the absence of locking by some numerical examples.

The main motivation for studying various coupled model is to adopt alternating parallelized algorithms to improve the calculation efficiency. Combining with suitable domain decomposition frameworks, we can partition the domain into several solvable subdomains with suitable transmission conditions\cite{mathew_domain_2008}. For example, the study in \cite{girault_domain_2011} introduced the model on the Steklov-Poincar\'e framework, which suggests various iterative algorithms for its solution, such as Dirichlet-Neumann method, Neumann-Neumann method and etc.. In this paper we adopt the Lagrange multiplier framework, that is, we introduced the normal stress on the interface from both subdomains as a unknown Lagrange multiplier variable $\boldsymbol{\lambda}$, and construct a saddle point system, where we can utilize the algorithms for iterating the Lagrange multiplier to approximate the transmission conditions, such as Uzawa's algorithm\cite{koko2010uzawa, ning2018uzawa}, non-overlapping Schwarz method\cite{magoules2004non, magoules2005optimal, antonietti2007schwarz}, finite element tearing and interconnecting(FETI) method\cite{lee2010mixed, lee2012analysis, lee2022feti, chu2025block} and etc..

In this paper, we adopt FETI method to solve the coupled poroelasticity and elasticity model with multiphysics finite element approximation, which will be presented in the rest of this paper organized in several sections. In \cref{sect:the_model} we present the coupled poroelasticity and elasticity model in a Lagrange multiplier framework, where the normal stress on the interface is introduced as the Lagrange multiplier term, then we reformulate the model by introducing new variables. Variational formulation and the corresponding semi-discretization of the reformulated model is given, in which a $\boldsymbol{P}_2$-$P_1$-$P_1$-$P_1$ finite element method and a $\boldsymbol{P}_2$-$P_1$ finite element method is chosen for the discretization, and we proved the existence and uniqueness of solutions to the model in each subdomain, and prescribed the equivalence of solution between coupled model and decoupled sub-models. In \cref{sect:algorithms}, a fully-discrete scheme is proposed by adopting the backward-Euler time-advancing method, which can be seen as a saddle point problem at each time step. Finally, in \cref{sect:computational_resutls}, several benchmark numerical experiments are provided to show the performance of the proposed approach and methods.

\section{The Coupled Poroelasticity and Elasticity Model}
\label{sect:the_model}

Consider a bounded polygonal domain $\Omega\subset\mathbb R^d(d=2,3)$ with the continuous boundary $\partial\Omega$. Suppose that $\Omega$ is partitioned into two non-overlapping subdomains $\Omega^P,\Omega^E$, which represents the quasi-static poroelastic reservoir layer and the elastic caprock layer respectively:
\[
\overline{\Omega}=\overline{\Omega^P\cup\Omega^E},\quad \Omega^P\cap\Omega^E=\emptyset.
\]
In this paper we use the appellation to each region in \cite{girault_domain_2011}, that is, the porous ``pay-zone" and non-porous ``nonpay-zone". The pay-zone $\Omega^P$ satisfies consolidation model for a linear elastic, homogeneous, isotropic, porous solid media saturated with a slightly compressible fluid. The nonpay-zone $\Omega^E$, a linear elasticity model will be adopted. For convenience when describing the finite element approximation, we assume that the domain $\Omega$ is a bounded polygon with boundary $\partial\Omega$ and interface $\varGamma=\overline{\Omega^{P}}\cap\overline{\Omega^E}$. For any variable $\mathcal G$, let $\mathcal G_{\mathcal D}:=\left.\mathcal G\right|_{\Omega^\mathcal D}\ (\mathcal D=P,E)$ denotes its restriction in subdomain $\Omega^{\mathcal D}$ (unless an extra description is given, we will use the superscript notation $\mathcal{D}$ to represent either subdomain through this paper). The normal vectors $\mathbf n_P$, $\mathbf n_E$, $\mathbf n_{P\to E}$ and $\mathbf n_{E\to P}$ are as shown in \cref{fig:fig-domain}.
\begin{figure}[!htbp]
\centering
\begin{tikzpicture}
\draw (-4, 0) -- (4, 0);
\draw (-4, 2) -- (4, 2) node [anchor=north west] {$\varGamma$};
\draw (-4, 3) -- (4, 3);
\draw (-4, 0) -- (-4, 3);
\draw (4, 0) -- (4, 3);
\draw (0, 0.8) node  {$\Omega^P$};
\draw (0, 2.3) node {$\Omega^E$};
\draw[-{Stealth[length=2mm,width=2mm]}] (0.75, 2) -- (0.75, 2.5) node [anchor=north west] {$\mathbf n_{P\to E}$};
\draw[-{Stealth[length=2mm,width=2mm]}] (4, 1) -- (4.5, 1) node [near end, below] {$\mathbf n_P$};
\draw[-{Stealth[length=2mm,width=2mm]}] (-1.25, 2) -- (-1.25, 1.5) node [anchor=south west] {$\mathbf n_{E\to P}$};
\draw[-{Stealth[length=2mm,width=2mm]}] (-4, 2.5) -- (-4.5, 2.5) node [near end, below] {$\mathbf n_E$};
\end{tikzpicture}
\caption{A top-bottom coupled structure of domains.}
\label{fig:fig-domain}
\end{figure}
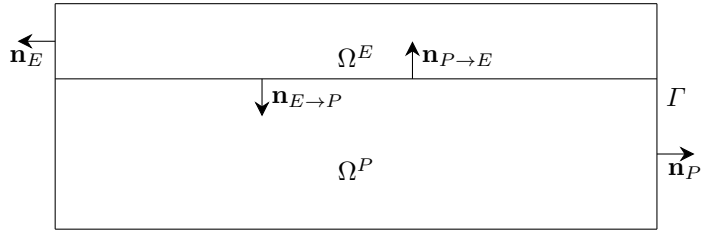

\subsection{Governing Equations}
In this paper, we consider the quasi-static poroelasticity model in $\Omega^P$ and elasticity in $\Omega^E$ interconnected by the interface $\varGamma$ \cite{phillips_coupling_2007, phillips_coupling_2007-1, girault_domain_2011}:
	Find $\mathbf u_p, \mathbf u_E, p$ for all $T>0$ such that 
	\begin{alignat}{3}
		\label{eqn:momentum_balance_P}    -\div\boldsymbol{\sigma}_P(\mathbf u_P) + \alpha\nabla p& = \mathbf f_P && \qquad\text{in }\Omega^P\times(0,T], \\
		\label{eqn:mass_balance}    \left( c_0p+\alpha\div\mathbf u \right)_t + \div\mathbf w(p)& = z && \qquad\text{in }\Omega^P\times(0,T], \\
		\label{eqn:momentum_balance_E}    -\div\boldsymbol{\sigma}_E(\mathbf u_E)& = \mathbf f_E && \qquad\text{in }\Omega^E\times(0,T],
	\end{alignat}
where that $\mathbf u_P, \mathbf u_E$ are the solid displacements respectively in $\Omega^P,\ \Omega^E$ and $p$ is the fluid pressure in $\Omega^P$. To close the model, we set the following boundary and initial conditions: 
\begin{alignat}{3}
    \label{BCO}\mathbf u_P&=\mathbf 0 &&\qquad\text{on }\varGamma^{P}_d\times(0,T], \\
    \label{BCT}\hat{\boldsymbol{\sigma}}_P(\mathbf{u}_P,p)\cdot\mathbf{n}_P&=\mathbf t_P&&\qquad \text{on }\varGamma^{P}_t\times(0,T], \\
    \label{BCP}p&=0&&\qquad\text{on }\varGamma^{P}_p\times(0,T], \\
    \label{BCF}\mathbf w(p)\cdot\mathbf{n}_P&=z_w&&\qquad \text{on }\varGamma^{P}_f\times(0,T], \\
    \label{BCOE}\mathbf u_E&=\mathbf 0&&\qquad\text{on }\varGamma^{E}_d\times(0,T], \\
    \label{BCTE}\boldsymbol{\sigma}_E(\mathbf{u}_E)\cdot\mathbf{n}_E&=\mathbf{t}_{E}&&\qquad \text{on }\varGamma^{E}_t\times(0,T], \\
    \label{ICU}\mathbf u(\mathbf x, 0)&=\mathbf u_0(\mathbf x)&&\qquad\text{in $\Omega$}, \\
    \label{ICP}p(\mathbf x, 0)&=p_0(\mathbf x)&&\qquad\text{in $\Omega^{P}$},
\end{alignat}
and the following transmission conditions are introduced for interconnecting the two subdomains $\Omega^P$ and $\Omega^E$:
\begin{alignat}{3}
 \label{TC1}&\text{the displacement on $\varGamma$ is continuous: }&&\qquad\mathbf{u}_P=\mathbf{u}_E, \\
 \label{TC2}&\text{the normal stress on $\varGamma$ is continuous: }&&\qquad-\hat{\boldsymbol{\sigma}}(\mathbf{u}_P,p)\cdot\mathbf n_{P\to E}=\boldsymbol{\sigma}_E(\mathbf u_E)\cdot\mathbf n_{E\to P}, \\
 \label{TC3}&\text{there is no fluid from $\Omega^P$ to $\Omega^E$: }&&\qquad\mathbf{w}\cdot\mathbf n_{P\to E}=0,
 \end{alignat}
where $\partial\Omega^{P}=\varGamma^{P}_d\cup\varGamma^{P}_t\cup\varGamma=\varGamma^{P}_p\cup\varGamma^{P}_f\cup\varGamma$ and $\partial\Omega^E=\varGamma^E_d\cup\varGamma^E_t\cup\varGamma$.

The problem \eqref{eqn:momentum_balance_P}-\eqref{eqn:momentum_balance_E} consists of a momentum balance in both $\Omega^P$ and $\Omega^E$ and a mass balance in $\Omega^P$, in which $\boldsymbol{\sigma}_\mathcal{D}(\mathbf u_{\mathcal D})$ represents the standard stress tensor from elasticity, and in $\Omega^P$ the Darcy's law holds for the volumetric flux $\mathbf w(p)$ relative linearly to $p$. For convenience, the stress and fluid pressure term in \eqref{eqn:momentum_balance_P} can be merged as the so-called total stress tensor noted by $\hat{\boldsymbol{\sigma}}_P(\mathbf u_P,p)$ (see \cref{tab:summary_of_constitutive_relations}). Parameters involved in the problem \eqref{eqn:momentum_balance_P}-\eqref{eqn:momentum_balance_E} are listed in \cref{tab:summary_of_involved_parameters}. In the assumption that the matrix is isotropic, the permeability tensor $\boldsymbol{K}$ is symmetric and uniformly positive definite, that is, for any choice of $\mathbf x$, the spectrum of $\boldsymbol{K}(\mathbf x)$ has both positive lower and upper bounds.

\begin{table}[H]
  \centering
  \caption{Summary of constitutive relations}
  \label{tab:summary_of_constitutive_relations}
  \begin{tabularx}{\textwidth}{XX}
    \toprule[1pt]
    Notation     &      Description    \\
    \midrule
    $\hat{\boldsymbol{\sigma}}_P(\mathbf u_P,p):=\boldsymbol{\sigma}_P(\mathbf u_P)-\alpha p\mathbf I$     &    Total stress tensor in $\Omega^P$    \\
    $\boldsymbol{\sigma}_{\mathcal D}(\mathbf u_{\mathcal D}):=\lambda_\mathcal{D}\div\mathbf u_\mathcal{D}+2\mu_\mathcal{D}\boldsymbol{\varepsilon}(\mathbf u_\mathcal{D})$     &    Effective stress tensor in $\Omega^{\mathcal{D}}$   \\
    $\mathbf w(p):=-\frac{\boldsymbol{K}}{\mu_f}\left(\nabla p-\rho_f\mathbf g\right)$    &    Volumetric fluid flux in $\Omega^P$ \\
    \bottomrule[1pt]
  \end{tabularx}
\end{table}

\begin{table}[!htbp]
  \centering
  \caption{Summary of involved parameters}
  \label{tab:summary_of_involved_parameters}
  \begin{tabularx}{\textwidth}{XX}
    \toprule[1pt]
    Notation     &      Description    \\
    \midrule
    $\lambda_\mathcal{D},\ \mu_\mathcal{D}$    &    Lam\'e constants    \\
    $\alpha$    &    Biot-Willis constant    \\
    $c_0$    &    Fluid specific storage coefficient    \\
    $\boldsymbol{K}$    &    Skeleton permeability tensor    \\
    $\mu_f$    &    Fluid viscosity    \\
    \bottomrule[1pt]
  \end{tabularx}
\end{table}

\subsection{Multiphysics Approach}
To reveal the underlying multiphysics process and overcome the locking phenomenon for the displacement variables of poroelasticity and elasticity and the pressure oscillation in computation, we set the so-called elastic pressure $\xi_\mathcal{D}$ and volumetric fluid content $\eta$ in the following form:
\begin{equation}
    \label{eqn:mreform}
    \xi_P=\alpha p-\lambda_P \epsilon_P,\quad \xi_E=-\lambda_E\epsilon_E,\quad \eta=c_0p+\alpha\epsilon_P,\quad\text{where}\quad \epsilon_\mathcal{D}:=\div\mathbf u_\mathcal{D}
\end{equation}
Since the new variables in \eqref{eqn:mreform} are linear combinations of the original variables, we can rewrite the original variables in terms of the new ones: \[
p=\kappa_1\xi_P+\kappa_2\eta,\quad \epsilon_P=\kappa_3\xi_P-\kappa_1\eta,\quad \epsilon_E=-\lambda_E^{-1}\xi_E,
\]where\[
\kappa_1=\frac{\alpha}{\alpha^2+c_0\lambda_P }, \quad \kappa_2=\frac{\lambda_P}{\alpha^2+c_0\lambda_P },\quad \kappa_3=\frac{c_0}{\alpha^2+c_0\lambda_P }.
\]

Using the above notations, we reformulate the problem \eqref{eqn:momentum_balance_P}-\eqref{eqn:momentum_balance_E} into the following nearly incompressible coupled poroelasticity and elasticity model:
	\begin{alignat}{3}
		\label{eqn:momentum_balance_reform}-2\mu_P\div\boldsymbol{\varepsilon}(\mathbf u_P)+\nabla\xi_P&=\mathbf f_{P}&&\qquad \text{in }\Omega^{P}\times(0,T], \\
		\label{eqn:divR1}\kappa_3\xi_P+\div\mathbf{u}_P-\kappa_1\eta&=0&&\qquad\text{in }\Omega^{P}\times(0,T], \\
		\label{eqn:mass_balance_reform}\eta_t+\div\mathbf w(p)&=z&&\qquad \text{in }\Omega^{P}\times(0,T], \\
		\label{eqn:momentum_balance_reform_E}-2\mu_E\div\boldsymbol{\varepsilon}(\mathbf u_E)+\nabla\xi_E&=\mathbf f_{E}&&\qquad \text{in }\Omega^{E}\times(0,T], \\
		\label{eqn:divR1_E}\lambda_E^{-1}\xi_E+\div \mathbf{u}_E&=0 &&\qquad\text{in }\Omega^{E}\times(0,T].
	\end{alignat}
	
Denote the total stress of the reformulated poroelasticity model by $\tilde{\boldsymbol{\sigma}}_\mathcal{D}(\mathbf u_\mathcal{D},\xi_\mathcal{D}):=2\mu_\mathcal{D}\boldsymbol{\varepsilon}(\mathbf u_\mathcal{D})-\xi_\mathcal{D}\mathbf I$, then we can rewrite the boundary conditions as 
\begin{alignat}{3}
    \label{BCOR}\mathbf u_P&=\mathbf 0&&\qquad\text{on }\varGamma^{P}_d\times(0,T], \\
    \label{BCTR}\tilde{\boldsymbol{\sigma}}(\mathbf{u}_P,\xi_P)\cdot\mathbf{n}_P&=\mathbf t_P&&\qquad \text{on }\varGamma_t^P\times(0,T], \\
    \label{BCPR}p&=p_D&&\qquad\text{on }\varGamma_p^{P}\times(0,T], \\
    \label{BCFR}\mathbf w(p)\cdot\mathbf{n}_P&=z_w&&\qquad \text{on }\varGamma_f^{P}\times(0,T], \\
    \label{BCORE}\mathbf u_E&=\mathbf 0&&\qquad\text{on }\varGamma^{E}_d\times(0,T], \\
    \label{BCTRE}\tilde{\boldsymbol{\sigma}}(\mathbf{u}_E,\xi_E)\cdot\mathbf{n}_E&=\mathbf{t}_{E}&&\qquad \text{on }\varGamma_t^E\times(0,T],
\end{alignat}
 and initial conditions:
\begin{alignat}{3}
    \label{ICUR}\mathbf u_P(\mathbf x, 0)&=\mathbf u_{P,0}(\mathbf x)&&\qquad\text{in $\Omega^{P}$}, \\
    \label{ICRE}\mathbf u_E(\mathbf x, 0)&=\mathbf u_{E,0}(\mathbf x)&&\qquad\text{in $\Omega^{E}$}, \\
    \label{ICPR}p(\mathbf x, 0)&=p_0(\mathbf x)&&\qquad\text{in $\Omega^{P}$}, \\
    \label{ICER}\eta(\mathbf x,0)&=\eta_0(\mathbf x)=c_0 p_0(\mathbf x)-\alpha\div\mathbf u_0(\mathbf x)&&\qquad\text{in $\Omega^{P}$}, \\
    \label{ICXR}\xi_P(\mathbf x,0)&=\xi_{P,0}(\mathbf x)=\alpha p_0(\mathbf x)-\lambda_P\div\mathbf u_0(\mathbf x)&&\qquad\text{in $\Omega^{P}$}, \\
    \label{ICXRE}\xi_E(\mathbf x,0)&=\xi_{E,0}(\mathbf x)=-\lambda_E\div\mathbf u_0(\mathbf x)&&\qquad\text{in $\Omega^{E}$},
\end{alignat}
and the transmission conditions:
\begin{alignat}{2}
	\label{TC1R}&\qquad\mathbf{u}_P=\mathbf{u}_E&&\qquad\text{on }\varGamma\times(0,T], \\
	\label{TC2R}&\qquad-\tilde{\boldsymbol{\sigma}}_P(\mathbf{u}_P,\xi_P)\cdot\mathbf n_{P\to E}=\tilde{\boldsymbol{\sigma}}_E(\mathbf u_E,\xi_E)\cdot\mathbf n_{E\to P}&&\qquad\text{on }\varGamma\times(0,T], \\
	\label{TC3R}&\qquad\mathbf{w}(p)\cdot\mathbf n_{P\to E}=0&&\qquad\text{on }\varGamma\times(0,T].
\end{alignat}

\subsection{Variational Formulations}
\label{sect:variational_formulations}

In this section we describe the variational formulation of the coupled poroelasticity and elasticity model with a definition of weak solutions to problems \eqref{eqn:momentum_balance_P}-\eqref{TC3} and corresponding reformulated problems \eqref{eqn:momentum_balance_reform}-\eqref{TC3R}. First for any Banach or Hilbert space $X$, we set $\boldsymbol{X}:=[X]^d$. We introduce the following standard notation for Sobolev spaces (see, e.g. \cite{boffi_mixed_2013}): \[
    \begin{aligned}
        \boldsymbol{X}^P&=\boldsymbol{H}^1_{0d}(\Omega^P):=\left\{\mathbf v_P\in\boldsymbol{H}^1(\Omega^P):\ \left.\mathbf v_P\right|_{\varGamma_d^P}=\mathbf 0\right\}, \\
        \boldsymbol{X}^E&=\boldsymbol{H}^1_{0d}(\Omega^E):=\left\{\mathbf v_E\in\boldsymbol{H}^1(\Omega^E):\ \left.\mathbf v_E\right|_{\varGamma_d^E}=\mathbf 0\right\}, \\
        W&=H^1_{0p}(\Omega^P):=\left\{\varphi\in H^1(\Omega^P):\ \left.\varphi\right|_{\varGamma_p^P}=0\right\},
    \end{aligned}
\]
 The inner products of $L^2(\Omega^{\mathcal D})$, $L^2(\varGamma_{t}^{\mathcal D})$ and $L^2(\varGamma_{f}^P)$ are denoted by $(\cdot,\cdot)_{\mathcal D}$, $\left<\cdot,\cdot\right>_{t,\mathcal D}$ and $\left<\cdot,\cdot\right>_f$, and we use the notation $\left<\cdot,\cdot\right>_\varGamma$ to denote the inner product of $L^2(\varGamma)$: \begin{alignat*}{3}
(f,g)_{\mathcal D}&:=\int_{\Omega^{\mathcal D}}fg~\mathrm d\mathbf x, &&\qquad f,g\in L^2(\Omega^{\mathcal D}) \\
\left<f,g\right>_{t,\mathcal D}&:=\int_{\varGamma_{t}^{\mathcal D}}fg~\mathrm ds,&&\qquad f,g\in L^2(\varGamma_{t}^{\mathcal D}), \\
\left<f,g\right>_f&:=\int_{\varGamma_{f}^{P}}fg~\mathrm ds,&&\qquad f,g\in L^2(\varGamma_{f}^{P}), \\
\left<f,g\right>_\varGamma&:=\int_{\varGamma}fg~\mathrm ds,&&\qquad f,g\in L^2(\varGamma).
\end{alignat*}
For the following description of the variational forms, let $\mathbf u_0\in\boldsymbol{H}^1(\Omega),\ \mathbf f_{P}\in\boldsymbol{L}^2(\Omega^P),\ \mathbf f_{E}\in\boldsymbol{L}^2(\Omega^E),\ \mathbf t_P\in\boldsymbol{L}^2(\varGamma_t^P),\ \mathbf{t}^E\in\boldsymbol{L}^2(\varGamma_t^E),\ p_0\in L^2(\Omega^P),\ h\in L^2(\Omega^P),\ z_w\in L^2(\varGamma_f^P) $ and $c_0$ is assumed to be positive in $\Omega^P$. 

\begin{problem}
    \label{def:vf-origin} Find $\mathbf u,p$ with $\mathbf u\in L^\infty(0,T;\boldsymbol{X})$ and $p\in L^\infty(0,T;L^2(\Omega^P))\cap L^2(0,T;H^1(\Omega^P))$ such that for $t\in [0,T]$ a.e.:
    \begin{alignat}{2}
        (\boldsymbol{\sigma}_P(\mathbf u_P),\boldsymbol{\varepsilon}(\mathbf v_P))_P&+(\boldsymbol{\sigma}_E(\mathbf u_E),\boldsymbol{\varepsilon}(\mathbf v_E))_E-\alpha(p,\div\mathbf v_P)_P+I_\varGamma \nonumber  \\
        &=(\mathbf f_{P},\mathbf v_P)_P+(\mathbf f_{P},\mathbf v_E)_E+\left<\mathbf t_P,\mathbf v_P\right>_{t,P}+\left<\mathbf t_E,\mathbf v_E\right>_{t,E} &&\qquad\forall\mathbf v_{\mathcal D}\in\boldsymbol{X}^{\mathcal D}, \\
        \left(\left( c_0p+\div\mathbf u_P \right)_t,\varphi\right)_P&+\frac{1}{\mu_f}\left(\boldsymbol{K}(\nabla p-\rho_f\mathbf g), \nabla\varphi\right)=(z,\varphi)_P+\left<z_w,\varphi\right>_f&&\qquad\forall\varphi\in W, \\
        &\mathbf u(0) = \mathbf u_0\ \text{in }\Omega,\quad p(0)=p_0\ \text{in }\Omega^P,
    \end{alignat}
    where the interface term $I_\varGamma$ takes the form: \[
        I_\varGamma=-\left<\hat{\boldsymbol{\sigma}}(\mathbf u_P,p)\cdot\mathbf n_{P\to E},\mathbf v_P\right>_\varGamma-\left<\boldsymbol{\sigma}_E(\mathbf u_E)\cdot\mathbf n_{E\to P},\mathbf v_E\right>_\varGamma.
    \]
\end{problem}
According to the normal stress continuity \eqref{TC2}, we define the Lagrange multiplier: \begin{equation}
    \boldsymbol{\lambda}=-\hat{\boldsymbol{\sigma}}(\mathbf{u}_P,p)\cdot\mathbf n_{P\to E}=\boldsymbol{\sigma}_E(\mathbf u_E)\cdot\mathbf n_{E\to P},
\end{equation}
then the interface term in \cref{def:vf-origin} can be written in \begin{equation}
    I_\varGamma=b_\varGamma(\mathbf v_P,\mathbf v_E;\boldsymbol{\lambda}),
\end{equation}
where \[
    b_\varGamma(\mathbf v_P,\mathbf v_E;\boldsymbol{\mu}) = \left<\mathbf v_P-\mathbf v_E,\boldsymbol{\mu}\right>_\varGamma.
\]
To guarantee the well-posedness of $b_\varGamma$, we require the following space\cite{mathew_domain_2008}: \[
    \boldsymbol{\lambda}\in\boldsymbol{\Lambda}:=\boldsymbol{H}^{-1/2}_{00}(\varGamma),
\]which is the dual space of $\boldsymbol{H}^{1/2}_{00}(\varGamma)$.

Next, we prescribe the variational formulation of the problem \eqref{eqn:momentum_balance_reform}-\eqref{TC3R}. For convenience, we introduce the following bilinear forms:
\begin{alignat}{2}
    \label{eqn:aD}a_\mathcal{D}(\mathbf u_\mathcal{D},\mathbf v_\mathcal{D})&=2\mu_\mathcal{D}(\boldsymbol{\varepsilon}(\mathbf u_\mathcal{D}),\boldsymbol{\varepsilon}(\mathbf v_\mathcal{D}))_\mathcal{D}, \\
    \label{eqn:bD}b_\mathcal{D}(\mathbf v_\mathcal{D}, \xi_\mathcal{D})&=-(\xi_\mathcal{D},\div\mathbf v_\mathcal{D})_\mathcal{D}, \\
    \label{eqn:af}a_f(p,q)&=\frac{1}{\mu_f}\left(\boldsymbol{K}\nabla p,\nabla q\right)_P.
\end{alignat}

\begin{problem}
    \label{def:vf-reformulated}
    Find $\mathbf u,\xi,\eta,p$ for $T>0$ with $\mathbf u\in L^\infty(0,T;\boldsymbol{X})$, $\xi\in L^2(0,T;L^2(\Omega))$, $\eta\in L^\infty(0,T;L^2(\Omega))\cap H^1(0,T;W')$ and $p\in L^\infty(0,T;L^2(\Omega))\cap L^2(0,T;W)$ for $t\in [0,T]$ such that
    \begin{alignat}{2}
        \label{eqn:vf1}&a_P(\mathbf u_P,\mathbf v_P)+b_P(\mathbf v_P,\xi_P)=(\mathbf f_{P},\mathbf v_P)_P+\left<\mathbf t_P,\mathbf v_P\right>_{t,P}+\left<\boldsymbol{\lambda},\mathbf v_P\right>_\varGamma &&\qquad\forall\mathbf v_P\in\boldsymbol{X}^P, \\
        \label{eqn:vf2}&b_P(\mathbf u_P,\zeta_P)-\kappa_3(\xi_P,\zeta_P)+\kappa_1(\eta,\zeta_P)_P=0&&\qquad\forall\zeta_P\in L^2(\Omega^P), \\
        \label{eqn:vf2-2}&\kappa_1(\xi_P,\psi)_P + \kappa_2(\eta, \psi)_P - (p, \psi)_P=0 &&\qquad\forall\psi\in L^2(\Omega^P), \\
        \label{eqn:vf3}&(\eta_t,q)+a_f(p,q)=(z, q)_P+\left<z_w, q\right>_f+\frac{\rho_f}{\mu_f}(\boldsymbol{K}\mathbf g,\nabla q)_P&&\qquad\forall q\in W, \\
        \label{eqn:vf4}&a_E(\mathbf u_E,\mathbf v_E)+b_E(\mathbf v_E,\xi_E)_E=(\mathbf f_{E},\mathbf v_E)_E+\left<\mathbf t_E,\mathbf v_E\right>_{t,E}-\left<\boldsymbol{\lambda},\mathbf v_E\right>_\varGamma&&\qquad\forall\mathbf v_E\in\boldsymbol{X}^E, \\
        \label{eqn:vf5}&b_E(\mathbf u_E,\zeta_E)-\lambda_E^{-1}(\xi_E,\zeta_E)_E =0&&\qquad\forall\zeta_E\in L^2(\Omega^E), \\
        \label{eqn:vf6}&\left<\mathbf u_P-\mathbf u_E, \boldsymbol{\mu}\right>_\varGamma=0 &&\qquad\forall\boldsymbol{\mu}\in\boldsymbol{\Lambda}, \\
        \label{eqn:vf7}&\mathbf u(0) = \mathbf u_{0},\quad p_h(0)=p_{0},\\
        \label{eqn:vf8}&\xi(0)=\xi_{0},\quad \eta(0)=\eta_{0}.
    \end{alignat}
\end{problem}

\subsection{Semi-discrete Scheme}

Let $\mathcal T_{h}^{\mathcal{D}} (I=P,E)$ be a suitable triangulation or rectangular partition for subdomain $\Omega^{\mathcal{D}}$, where $\overline{\Omega^{\mathcal{D}}}=\bigcup_{K\in\mathcal T_h^{\mathcal{D}}}\overline K$, and let $(\boldsymbol{X}_h^{\mathcal{D}}, M_h^{\mathcal{D}})$ be the space of suitable mixed finite element functions pair. 
We can naturally note the corresponding triangulation and finite element pair from the above definitions by $\mathcal T_h=\mathcal T_h^P\cup\mathcal T_h^E,\ \boldsymbol{X}_h=\boldsymbol{X}_h^P\oplus\boldsymbol{X}_h^E$ and $M_h=M_h^P\oplus M_h^E$. Note that finite element space $W_h$ for the fluid pressure $p$ and volumetric fluid content $\eta$ in $\Omega^P$ can be chosen independently, any choice that $W_h\supset M_h^P$ is acceptable. In this section we consider the convenient choice that $W_h=M_h^P$. In addition, let $\boldsymbol{\Lambda}_h$ denotes the finite element space of the Lagrange multiplier $\boldsymbol{\lambda}$. In this paper, we adopt  $\boldsymbol{P}_k$-$P_l$ element $(\boldsymbol{X}_h^{\mathcal{D}}, M_h^{\mathcal{D}})$ for $\mathbf u_{\mathcal{D}},\ \xi_{\mathcal{D}}$:
\[
    \begin{alignedat}{2}
        &\boldsymbol{X}_h^{\mathcal{D}}:=\left\{\mathbf v_h\in\boldsymbol{C}^0(\overline{\Omega^{\mathcal{D}}}):\ \left.\mathbf v_{\mathcal{D},h}\right|_K\in \boldsymbol{P}_k(K)\quad \forall K\in\mathcal T_h^{\mathcal{D}}\right\}, \\
        &M_h^{\mathcal{D}}:=\left\{\varphi_h\in C^0(\overline{\Omega^{\mathcal{D}}}):\ \left.\varphi_{\mathcal{D},h}\right|_K\in P_l(K)\quad \forall K\in\mathcal T_h^{\mathcal{D}}\right\}. 
    \end{alignedat}
\]
where $k, l$ will be specifically given in the test cases. In the latter description, we will always use ($\boldsymbol{P}_k$-$P_l$-$P_l$-$P_l$, $\boldsymbol{P}_k$-$P_l$) to denote that we use $\boldsymbol{P}_k$-$P_l$-$P_l$-$P_l$ finite element space for poroelasticity and $\boldsymbol{P}_k$-$P_l$ finite element space for elasticity. Now, we give the finite element discretized variational formulation in the following:
\begin{problem}
    \label{def:vf-reformulated-discretized}
    Find $\mathbf u_h,\xi_h,\eta_h,p_h$ for $T>0$ with $\mathbf u_h\in L^\infty(0,T;\boldsymbol{X}_h)$, $\xi_h\in L^\infty(0,T;M_h)$, $\eta_h\in L^\infty(0,T;M_h^P)\cap H^1(0,T;W_h')$ and $p_h\in L^\infty(0,T;M_h)\cap L^2(0,T;W_h)$ for $t\in [0,T]$ such that
    \begin{alignat}{3}
        \label{eqn:vf1h}&a_P(\mathbf u_{P,h},\mathbf v_{P,h})+b_P(\mathbf v_{P,h},\xi_{p,h})=(\mathbf f_{P},\mathbf v_{P,h})_P+\left<\mathbf t_P,\mathbf v_{P,h}\right>_{t,P}-\left<\boldsymbol{\lambda}_h,\mathbf v_{P,h}\right>_\varGamma &&\qquad\forall\mathbf v_{P,h}\in\boldsymbol{X}_h^P, \\
        \label{eqn:vf2h}&b(\mathbf u_{P,h},\zeta_{P,h})-\kappa_3(\xi_{P,h},\zeta_{P,h})_P+\kappa_1(\eta_h,\zeta_{P,h})_P=0&&\qquad\forall\zeta_{P,h}\in M_h^P, \\
        \label{eqn:vf2-2h}&\kappa_1(\xi_{P,h},\psi_h)_P+\kappa_2(\eta_h,\psi_h)_P-(p_h,\psi_h)=0&&\qquad\forall\psi_h\in W_h \\
        \label{eqn:vf3h}&(\eta_{h,t}, q_h)_P+a_f(p_h, q_h)=(z, q_h)_P+\left<z_w, q_h\right>_f+\frac{\rho_f}{\mu_f}(\boldsymbol{K}\mathbf g,\nabla q_h)_P&&\qquad\forall q_h\in W_h, \\
        \label{eqn:vf4h}&a_E(\mathbf u_{E,h},\mathbf v_{E,h})+b_E(\mathbf v_{E,h},\xi_{E,h})=(\mathbf f_{P},\mathbf v_{E,h})_E+\left<\mathbf t_E,\mathbf v_{E,h}\right>_{t,E}+\left<\boldsymbol{\lambda}_h,\mathbf v_{E,h}\right>_\varGamma&&\qquad\forall\mathbf v_{E,h}\in\boldsymbol{X}_h^E, \\
        \label{eqn:vf5h}&b_E(\mathbf u_{E,h},\zeta_{E,h})-\lambda_E^{-1}(\xi_{E,h},\zeta_{E,h})_E=0&&\qquad\forall\zeta_{E,h}\in M_h^E, \\
        \label{eqn:vf6h}&\left<\mathbf u_{P,h}-\mathbf u_{E,h}, \boldsymbol{\mu}_h\right>_\varGamma=0 &&\qquad\forall\boldsymbol{\mu}_h\in\boldsymbol{\Lambda}_h, \\
        \label{eqn:vf7h}&\mathbf u_h(0) = \mathbf u_{0,h},\quad p_h(0)=p_{0,h},\\
        \label{eqn:vf8h}&\xi_h(0)=\xi_{0,h},\quad \eta_h(0)=\eta_{0,h}.
    \end{alignat}
    where $\mathbf u_{0,h}, p_{0,h}, \xi_{0,h}$ and $\eta_{0,h}$ are the projections of the initial conditions respectively.
\end{problem}

\begin{remark}
    In practical actuality, we do not need to choose finite element spaces that the inf-sup condition holds, that is, the space $\boldsymbol{X}_h^{\mathcal D}$ and $M_h^{\mathcal D}$ with $k>l$, which can be demonstrated by the proof of existence and uniqueness of solutions of the following matrix-vector forms in Theorem \ref{thm:ex-uni}. 
\end{remark}

\subsection{Fully-discrete Scheme}
\label{sect:fully-discrete-scheme}

In this section, we will describe the fully-discrete scheme by the backward Euler method for time discretization. To do that, we assume that the time interval $[0,T]$ is divided uniformly into $N$ segments by setting $t_n=n\tau$, where $\tau=T/N$ is the time step size and $n=0,1,\cdots,N$, and let $\mathcal G^{n}:=\mathcal G(t_n)$, where $\mathcal G$ denotes any variable. 

We give the following fully discretized variational formulation in the following:

\begin{problem}
    \label{def:vf-reformulated-fully-discretized}
    Find $\mathbf u_h^n,\xi_h^n,\eta_h^n,p_h^n$ with $\mathbf u_h^n\in \boldsymbol{X}_h$, $\xi_h^n\in M_h$, $\eta_h^n\in W_h$ and $p_h^n\in W_h$ for $n=1,\cdots,N$ such that:
    \begin{alignat}{3}
        \label{eqn:vf1ht}&a_P(\mathbf u_{P,h}^n,\mathbf v_{P,h})+b_P(\mathbf v_{P,h},\xi_{p,h}^n)=(\mathbf f_{P}(t_n),\mathbf v_{P,h})_P+\left<\mathbf t_P(t_n),\mathbf v_{P,h}\right>_{t,P}-\left<\boldsymbol{\lambda}_h^n,\mathbf v_{P,h}\right>_\varGamma &&\qquad\forall\mathbf v_{P,h}\in\boldsymbol{X}_h^P, \\
        \label{eqn:vf2ht}&b(\mathbf u_{P,h}^n,\zeta_{P,h})-\kappa_3(\xi_{P,h}^n,\zeta_{P,h})_P+\kappa_1(\eta_h^n,\zeta_{P,h})_P=0&&\qquad\forall\zeta_{P,h}\in M_h^P, \\
        \label{eqn:vf2-2ht}&\kappa_1(\xi_{P,h}^n,\psi_h)_P+\kappa_2(\eta_h^n,\psi_h)_P-(p_h^n,\psi_h)=0&&\qquad\forall\psi_h\in W_h \\
        \label{eqn:vf3ht}&(\partial_t\eta_h^n, q_h)_P+a_f(p_h^n, q_h)=(z(t_n), q_h)_P+\left<z_w(t_n), q_h\right>_f+\frac{\rho_f}{\mu_f}(\boldsymbol{K}\mathbf g,\nabla q_h)_P&&\qquad\forall q_h\in W_h, \\
        \label{eqn:vf4ht}&a_E(\mathbf u_{E,h}^n,\mathbf v_{E,h})+b_E(\mathbf v_{E,h},\xi_{E,h}^n)=(\mathbf f_{P}(t_n),\mathbf v_{E,h})_E+\left<\mathbf t_E(t_n),\mathbf v_{E,h}\right>_{t,E}+\left<\boldsymbol{\lambda}_h^n,\mathbf v_{E,h}\right>_\varGamma&&\qquad\forall\mathbf v_{E,h}\in\boldsymbol{X}_h^E, \\
        \label{eqn:vf5ht}&b_E(\mathbf u_{E,h}^n,\zeta_{E,h})-\lambda_E^{-1}(\xi_{E,h}^n,\zeta_{E,h})_E=0&&\qquad\forall\zeta_{E,h}\in M_h^E, \\
        \label{eqn:vf6ht}&\left<\mathbf u_{P,h}^n-\mathbf u_{E,h}^n, \boldsymbol{\mu}_h\right>_\varGamma=0 &&\qquad\forall\boldsymbol{\mu}_h\in\boldsymbol{\Lambda}_h, \\
        \label{eqn:vf7ht}&\mathbf u_h^0 = \mathbf u_{0,h},\quad p_h^0=p_{0,h},\\
        \label{eqn:vf8ht}&\xi_h^0=\xi_{0,h},\quad \eta_h^0=\eta_{0,h}.
    \end{alignat}
    where the backward difference quotient $\partial_t\eta_h^n$ is defined as \begin{equation}
        \label{eqn:backward-difference}
        \partial_t\eta_h^n=\frac{\eta_h^n-\eta_h^{n-1}}{\tau},
    \end{equation}
    and where $\mathbf u_{0,h}, p_{0,h}, \xi_{0,h}$ and $\eta_{0,h}$ are the projections of the initial conditions respectively.
\end{problem}

Applying \eqref{eqn:backward-difference} to \eqref{eqn:vf3ht} yields the following: \begin{align}
        &-(\eta_h^n, q_h)_P-\tau a_f(p_h^n, q_h)=-(\eta_h^{n-1} + \tau z(t_n), q_h)_P-\tau\left<z_w(t_n), q_h\right>_f-\frac{\tau\rho_f}{\mu_f}(\boldsymbol{K}\mathbf g,\nabla q_h)_P, &&\forall q_h\in W_h. \label{eqn:vf3ht-reformulated}
\end{align}which is the finite element formulation of an elliptic equation. 

We now construct the linear systems at each time step. Let $\boldsymbol{X}_h^\mathcal{D}=\operatorname{span}\{\boldsymbol{\theta}_{\mathcal{D},i}\}_{i=1}^{N^\mathcal{D}_\mathbf{u}}$, $M_h^\mathcal{D}=\operatorname{span}\{\phi_{\mathcal D,i}\}_{i=1}^{N^\mathcal{D}_\xi}$, $W_h=\operatorname{span}\{\phi_i\}_{i=1}^{N_p}$ and $\boldsymbol{\Lambda}_h=\operatorname{span}\{\boldsymbol{\upsilon}_i\}_{i=1}^{N_{\boldsymbol{\lambda}}}$, where $N^\mathcal{D}_\mathbf{u}$, $N^\mathcal{D}_\xi$, $N_p$ and $N_{\boldsymbol{\lambda}}$ are the degrees of freedoms to $\boldsymbol{X}_h^\mathcal{D}$, $M_h^\mathcal{D}$, $W_h$ and $\boldsymbol{\Lambda}_h$ respectively. According the settings of finite element space for each variables, we write $\mathbf u_{\mathcal{D},h}(t)=\sum_{i=1}^{N_\mathbf{u}^\mathcal{D}}\mathbf u_{\mathcal{D},i}(t)\boldsymbol{\theta}_{\mathcal{D},i}$, $\xi_{\mathcal{D},h}(t)=\sum_{i=1}^{N_\xi^\mathcal{D}}\xi_{\mathcal{D},i}(t)\phi_{\mathcal{D},i}$, $\eta_h(t)=\sum_{i=1}^{N_p}\eta_i(t)\phi_{P,i}$, $p_h(t)=\sum_{i=1}^{N_p}p_i(t)\phi_{P,i}$ and $\boldsymbol{\lambda}_h=\sum_{i=1}^{N_{\boldsymbol{\lambda}}}\boldsymbol{\lambda}_i(t)\boldsymbol{\upsilon}_i$. The fully-discretized system deriving from \cref{def:vf-reformulated-fully-discretized} can be written in the following matrix-vector form: for each time step $n=1,2,\cdots,N$, find $\mathbf u_{\mathcal D,h}^n$, $\xi_{\mathcal D,h}^n$, $\eta_h^n$ and $p_h^n$ such that
\begin{equation}
    \label{eqn:linalg-system}
    \begin{bmatrix}
        \mathcal A_P    &    \mathcal B_P^T    &    \mathcal O    &    \mathcal O    &    \mathcal O    &    \mathcal O    &    \mathcal H_P^T    \\
        \mathcal B_P    &    -\kappa_3\mathcal R_P    &    \kappa_1\mathcal R_P    &    \mathcal O    &    \mathcal O    &    \mathcal O    &    \mathcal O    \\
        \mathcal O    &    \kappa_1\mathcal R_P    &    \kappa_2\mathcal R_P    &   -\mathcal R_P   &    \mathcal O    &    \mathcal O    &    \mathcal O    \\
        \mathcal O    &     \mathcal O    &       -\mathcal R_P&    -\tau\mathcal A_f    &    \mathcal O    &    \mathcal O    &    \mathcal O    \\
        \mathcal O    &    \mathcal O    &    \mathcal O    &    \mathcal O    &    \mathcal A_E    &    \mathcal B_E^T    &    -\mathcal H_E^T    \\
        \mathcal O    &     \mathcal O    &       \mathcal O&    \mathcal O    &       \mathcal B_E&    -\lambda_E^{-1}\mathcal R_E    &    \mathcal O    \\
        \mathcal H_P    &    \mathcal O    &    \mathcal O    &    \mathcal O    &       -\mathcal H_E &    \mathcal O    &    \mathcal O
    \end{bmatrix}\begin{bmatrix}
        \mathbf u_{P,h}^n \\
        \xi_{P,h}^n \\
        \eta^n \\
        p_h^n \\
        \mathbf u_{E,h}^n \\
        \xi_{E,h}^n \\
        \boldsymbol{\lambda}_h^n
    \end{bmatrix}=\begin{bmatrix}
        \mathcal F_P^n    \\
        \mathcal O      \\
        \mathcal O      \\
        -\mathcal R_P\eta_h^{n-1} - \tau\mathcal Z^n      \\
        \mathcal F_E^n    \\
        \mathcal O      \\
        \mathcal O
    \end{bmatrix},
\end{equation}
where the corresponding representations of the above block matrices and vectors are
\begin{align*}
    \left[\mathcal A_P\right]_{ij}&=a_P(\boldsymbol{\theta}_{P,i},\boldsymbol{\theta}_{P,j}), && \left[\mathcal A_E\right]_{ij}=a_E(\boldsymbol{\theta}_{E,i},\boldsymbol{\theta}_{E,j}), \\
    \left[\mathcal B_P\right]_{ij}&=b_P(\boldsymbol{\theta}_{P,j},\phi_{P,i}), && \left[\mathcal B_E\right]_{ij}=b_E(\boldsymbol{\theta}_{E,j},\phi_{E,i}), \\
    \left[\mathcal H_P\right]_{ij}&=\left<\boldsymbol{\theta}_{P,j},\boldsymbol{\upsilon}_i\right>_\varGamma, && \left[\mathcal H_E\right]_{ij}=\left<\boldsymbol{\theta}_{E,j},\boldsymbol{\upsilon}_i\right>_\varGamma, \\
    \left[\mathcal R_P\right]_{ij}&=(\phi_{P,i},\phi_{P,j}), && \left[\mathcal R_E\right]_{ij}=(\phi_{E,i},\phi_{E,j}), \\
    \left[\mathcal F_P\right]_{i}&=(\mathbf f_P(t_n),\boldsymbol{\theta}_{P,i}) + \left<\mathbf t_P(t_n),\boldsymbol{\theta}_{P,i}\right>_{t,P}, && \left[\mathcal F_E\right]_{i}=(f_E(t_n),\boldsymbol{\theta}_{E,i})+\left<\mathbf t_E(t_n),\boldsymbol{\theta}_{E,i}\right>,
\end{align*}
\begin{equation*}
    \left[\mathcal Z^n\right]_{i}=(z(t_n),\phi_{P,i})_P+\left<z_w(t_n),\phi_{P,i}\right>_f+\frac{\rho_f}{\mu_f}(\boldsymbol{K}\mathbf g,\nabla\phi_{P,i})_P.
\end{equation*}

\begin{remark}
    \label{rem:inf-sup}
    The fully-discretized scheme described in \cref{def:vf-reformulated-fully-discretized} is not optimal for parallel solving. Assume that the degree of freedoms for $\mathcal T_P$ and $\mathcal T_E$ are the same, since the reformulated poroelasticity subproblem \eqref{eqn:vf1ht}-\eqref{eqn:vf3ht} has extra two variables $\eta$, $p$ compared to the nearly-incompressible elasticity subproblem \eqref{eqn:vf4ht}-\eqref{eqn:vf5ht}, the computational cost for solving the poroelasticity subproblem is much higher than that of the elasticity subproblem, which leads to load imbalance in parallel computing. To overcome this issue, we need to adopt the explicit time discretization for the equation of $\eta$, $p$ in \eqref{eqn:vf3h} in order to decouple the momentum balance equations \eqref{eqn:vf1h}-\eqref{eqn:vf2h} and the mass balance equation \eqref{eqn:vf2-2h}-\eqref{eqn:vf3h}. However, classical linear explicit schemes may lead to too strict stability condition, which forces us to choose an unacceptably small time step size. To this end, we can consider the exponential integrator\cite{hochbruck2010exponential} for a accuracy-preserving decoupling.
\end{remark}

\section{The FETI Method}
\label{sect:algorithms}

In this section we consider implementing the FETI method to the coupled poroelasticity and elasticity problem. The target of discussing the FETI algorithm is to find the displacement on the interface $\left.\mathbf u\right|_\varGamma$ such that the poroelasticity and elasticity subproblems in their corresponding subdomains can be solved independently as generalized single-region problems, which has the potential for parallelization. To see this, we first reduce the saddle point system \eqref{eqn:linalg-system} to a smaller one for solving interface displacement $\left.\mathbf u_h^n\right|_\varGamma$ and Lagrange multiplier $\boldsymbol{\lambda}_h^n$ only. For the reduced system, since the difficulty for explicitly showing the stiffness matrix, we consider a preconditioned conjugate gradient (PCG) method with the FETI Dirichlet preconditioner(cf. \cite{mathew_domain_2008}).

\subsection{The FETI Algorithm and Preconditioner} Before starting the description, we first reorder some variables so that each subproblem has the form of saddle-point system. We set\label{sect:the_feti_algorithm_and_preconditioner}\begin{align*}
    \mathcal{U}_{P,h}(t)=\begin{bmatrix}
        \mathbf u_{P,h}(t)	\\	\eta_h(t)
    \end{bmatrix},\quad \mathcal{U}_{E,h}(t)=
        \mathbf u_{E,h}(t),
    \quad \mathcal{P}_{P,h}(t)=\begin{bmatrix}
        \xi_{P,h}(t)	\\	p_h(t)
    \end{bmatrix},	\quad \mathcal{P}_{E,h}(t)=
        \xi_{E,h}(t),
\end{align*}, and then we rewrite \eqref{eqn:linalg-system} in the following form:
\begin{equation}\label{eqn:linalg-system-reformulated}
    \mathcal K\begin{bmatrix}
        \mathcal U_{P,h}^{n}    & \mathcal P_{P,h}^{n}    &    \mathcal U_{E,h}^{n}    & \mathcal P_{E,h}^{n}    &    \boldsymbol{\lambda}_h^n
    \end{bmatrix}^T=\begin{bmatrix}
        \mathcal F_{P,*}^n    &
        \mathcal Z_{*}^n    &
        \mathcal F_{E,*}^n    &
        \mathcal O      &
        \mathcal O
    \end{bmatrix}^T,
\end{equation}
where \[
    \mathcal K=\begin{bmatrix}
        \mathcal A_P^*    &    \mathcal (B_P^*)^T    &    \mathcal O    &    \mathcal O    &    \mathcal (H_P^*)^T    \\
        \mathcal B_P^*    &    -\mathcal C_P^*    &    \mathcal O    &    \mathcal O    &    \mathcal O    \\
            \mathcal O    &    \mathcal O    &    \mathcal A_E    &    \mathcal B_E^T    &    \mathcal (H_E^*)^T    \\
            \mathcal O    &     \mathcal O    &       \mathcal B_E&    -\lambda_E^{-1}\mathcal R_E    &    \mathcal O    \\
            \mathcal H_P^*    &    \mathcal O    &       \mathcal H_E^* &    \mathcal O    &    \mathcal O
    \end{bmatrix}, \mathcal F_{P,*}^n=\begin{bmatrix}
        \mathcal F_P^n    \\
        \mathcal O
    \end{bmatrix},\quad \mathcal Z_{*}^n=\begin{bmatrix}
        \mathcal O    \\
        -\mathcal R_P\eta_h^{n-1} - \tau\mathcal Z^n
    \end{bmatrix}, \mathcal F_{E,*}^n=\begin{bmatrix}
        \mathcal F_E^n    \\
        \mathcal O
    \end{bmatrix},
\] and where \[
    \mathcal A_P^*=\begin{bmatrix}
        \mathcal A_P    &    \mathcal O    \\
        \mathcal O    &    \kappa_2\mathcal R_P
    \end{bmatrix},\quad \mathcal B_P^*=\begin{bmatrix}
        \mathcal B_P    &    \kappa_1\mathcal R_P    \\
        \mathcal O    &    -\mathcal R_P
    \end{bmatrix},\quad \mathcal C_P^*=\begin{bmatrix}
        \kappa_3\mathcal R_P    &    \mathcal O    \\
        \mathcal O    &    \mathcal A_f
    \end{bmatrix},\quad \mathcal H_P^*=\begin{bmatrix}
        \mathcal H_P    &    \mathcal O
    \end{bmatrix}.
\]

We first reorder the vector of degrees of freedom $\mathcal U_{\mathcal D,h}^n$ as $\left(\mathcal U_{\mathcal D,h,I}^n,\ \mathcal U_{\mathcal D,h,B}^n\right)^T$ and corresponding load vectors as $\mathcal F_{\mathcal D}^n=\left(\mathcal F_{\mathcal D,I}^n,\ \mathcal F_{\mathcal D,B}^n\right)^T$ , which yields from \eqref{eqn:linalg-system-reformulated} the following reordered system: \begin{equation}\label{eqn:reordered-system}
    \begin{bmatrix}
        \mathcal A_{P,II}^*    &    (\mathcal B_{P,I}^*)^T    &    \mathcal (A_{P,IB}^*)^T    &    \mathcal O    &    \mathcal O    &      \mathcal O    &    \mathcal O   \\
        \mathcal B_{P,I}^*    &    -\mathcal C_P^*    &    \mathcal (B_{P,B}^*)^T    &    \mathcal O    &    \mathcal O    &     \mathcal O    &    \mathcal O    \\
        \mathcal A_{P,IB}^*    &    \mathcal B_{P,B}^*    &    \mathcal A_{P,BB}^*    &    \mathcal O    &    \mathcal O    &     \mathcal O    &    \mathcal H_{P}^T    \\
            \mathcal O    &    \mathcal O    &    \mathcal O    &    \mathcal A_{E,II}    &    \mathcal B_{E,I}^T    &    \mathcal A_{E,IB}^T    &    \mathcal O    \\
            \mathcal O       &    \mathcal O     &     \mathcal O    &       \mathcal B_{E,I}&    -\mathcal C_E    &    \mathcal B_{E,B}^T    &    \mathcal O    \\
            \mathcal O    &    \mathcal O    &    \mathcal O    &    \mathcal A_{E,IB}    &    \mathcal B_{E,B}    &    \mathcal A_{E,BB}     &        \mathcal H_{E}^T  \\
            \mathcal O    &    \mathcal O    &    \mathcal H_{P}    &    \mathcal O    &       \mathcal O    &       \mathcal H_{E} &    \mathcal O
    \end{bmatrix}\begin{bmatrix}
        \mathcal U_{P,h,I}^n    \\
        \mathcal P_{P,h}^n    \\
        \mathcal U_{P,h,B}^n    \\
        \mathcal U_{E,h,I}^n    \\
        \mathcal P_{E,h}^n    \\
        \mathcal U_{E,h,B}^n    \\
        \boldsymbol{\lambda}_h^n
    \end{bmatrix}=\begin{bmatrix}
        \mathcal F_{P,I}^n    \\
        \mathcal Z^n_*    \\
        \mathcal F_{P,B}^n    \\
        \mathcal F_{E,I}^n    \\
        \mathcal O    \\
        \mathcal F_{P,B}^n    \\
        \mathcal O
    \end{bmatrix}.
\end{equation}
Substituting this expression into the second block row of \eqref{eqn:reordered-system} yields a reduced saddle point system for determining $\mathcal U_{\mathcal D,h,B}^n$, that is, the reduced saddle point system: \begin{equation}\label{eqn:reduced-system}
    \begin{bmatrix}
        \mathcal S_P    &    \mathcal O    &    \mathcal H_{P}^T    \\
        \mathcal O    &    \mathcal S_E    &    \mathcal H_{E}^T    \\
        \mathcal H_{P}    &    \mathcal H_{E}    &\mathcal O
    \end{bmatrix}\begin{bmatrix}
        \mathcal U_{P,h,B}^n    \\
        \mathcal U_{E,h,B}^n    \\
        \boldsymbol{\lambda}_h^n
    \end{bmatrix} = \begin{bmatrix}
       \tilde{\mathcal F}_{P,B}^n    \\
        \tilde{\mathcal F}_{E,B}^n    \\
        \mathcal O
    \end{bmatrix},
\end{equation}
where \[
\begin{aligned}
    \mathcal S_{P}&=\mathcal A_{P,BB}^*-\begin{bmatrix}
        \mathcal A_{P,IB}^*    &    \mathcal B_{P,B}^*
    \end{bmatrix}\begin{bmatrix}
        \mathcal A_{P,II}^*    &    (\mathcal B_{P,I}^*)^T    \\
        \mathcal B_{P,I}^*    &    -\mathcal C_P^*
    \end{bmatrix}^{-1}\begin{bmatrix}
        (\mathcal A_{P,IB}^*)^T    \\    (\mathcal B_{P,B}^*)^T
    \end{bmatrix}, \\
    \mathcal S_{E}&=\mathcal A_{E,BB}-\begin{bmatrix}
        \mathcal A_{E,IB}    &    \mathcal B_{E,B}
    \end{bmatrix}\begin{bmatrix}
        \mathcal A_{E,II}    &    \mathcal B_{E,I}^T    \\
        \mathcal B_{E,I}    &    -\mathcal C_E
    \end{bmatrix}^{-1}\begin{bmatrix}
        \mathcal A_{E,IB}^T    \\    \mathcal B_{E,B}^T
    \end{bmatrix},
    \end{aligned}
\] and \[
    \begin{aligned}
    \tilde{\mathcal F}_{P,B}^n&=\tilde{\mathcal F}_{P,B}^n - \mathcal A_{P,BB}-\begin{bmatrix}
        \mathcal A_{P,IB}    &    \mathcal B_{P,B}
    \end{bmatrix}\begin{bmatrix}
        \mathcal F_{P,I}^n    \\
        \mathcal \mathcal Z^n_*
    \end{bmatrix},    \\
    \tilde{\mathcal F}_{E,B}^n&=\tilde{\mathcal F}_{E,B}^n - \mathcal A_{E,BB}-\begin{bmatrix}
        \mathcal A_{E,IB}    &    \mathcal B_{E,B}
    \end{bmatrix}\begin{bmatrix}
        \mathcal F_{E,I}^n    \\
        \mathcal O
    \end{bmatrix}.
    \end{aligned}
\]

We now describe the existence and uniqueness of solutions of \eqref{eqn:reordered-system} in the following.

\begin{theorem}\label{thm:ex-uni}
    The matrix-vector form \eqref{eqn:linalg-system}, reordered into \eqref{eqn:reordered-system}, is uniquely solvable.
\end{theorem}

\begin{proof}
        In actuality, according to the reduction from \eqref{eqn:reordered-system} to \eqref{eqn:reduced-system}, the stiffness matrix in \eqref{eqn:reordered-system} is transformed into the following diagonal form: \[
    \operatorname{diag}\left\{
        \begin{bmatrix}
        \mathcal A_{P,II}^*    &    (\mathcal B_{P,I}^T)^*    \\
        \mathcal B_{P,I}^*    &    -\mathcal C_P^*
    \end{bmatrix},
    \begin{bmatrix}
        \mathcal A_{E,II}    &    \mathcal B_{E,I}^T    \\
        \mathcal B_{E,I}    &    -\mathcal C_E
    \end{bmatrix},\begin{bmatrix}
        \mathcal S_P    &    \mathcal O    &    \mathcal H_{P}^T    \\
        \mathcal O    &    \mathcal S_E    &    \mathcal H_{E}^T    \\
        \mathcal H_{P}    &    \mathcal H_{E}    &\mathcal O
    \end{bmatrix}
    \right\}.
\] We need only to prove respectively that the above three diagonal blocks are all non-singular and then the correspondingly subequations are all uniquely solvable.

        First we prove that the Schur complement $\mathcal S_P$, $\mathcal S_E$ are both symmetric and positive definite. From the Sylvester's law of inertia, the number of positive, zero and negative eigenvalues do not change through elementary transformations, thus we can verify that the Schur complement $\mathcal S_P$ and $\mathcal S_E$ are symmetric and positive definite. 

        Then we need only to prove that the reduced system \eqref{eqn:reduced-system}, which is a standard saddle point system, is uniquely solvable. Since the block diagonal $\operatorname{diag}\left(\mathcal S_P,\mathcal S_E\right)$ is symmetric and positive definite and $\left[\mathcal H_P,\mathcal H_E\right]$ is surjective, the reduced system \eqref{eqn:reduced-system} is uniquely solvable according to Corollary 3.2.1 in \cite{boffi_mixed_2013}.

        Once we have solved the interface displacement $\mathcal U_{\mathcal D,h,B}$ and Lagrange multiplier $\boldsymbol{\lambda}_h^n$, the solution of \eqref{eqn:linalg-system} can be obtained by solving \eqref{eqn:reduced-system} for $\mathcal U_{P,h,B}^n,\ \mathcal U_{E,h,B}^n,\ \boldsymbol{\lambda}_h^n$ and then consequently solve \[\begin{aligned}
        \begin{bmatrix}
             \mathcal A_{P,II}^*   &    (\mathcal B_{P,I}^*)^T    \\
             \mathcal B_{P,I}^*    &    \mathcal C_P^*
        \end{bmatrix}\begin{bmatrix}
            \mathcal U_{P,h,I}^n    \\
            \mathcal P_{P,h}^n    \\
        \end{bmatrix}&=\begin{bmatrix}
            \mathcal F_{P,I}^n    \\
            \mathcal I_\tau\mathcal Z^n
        \end{bmatrix} - \begin{bmatrix}
            (\mathcal A_{P,IB}^*)^T    \\    (\mathcal B_{P,B}^*)^T
        \end{bmatrix}\mathcal U_{P,h,B}^n  - \mathcal H_{P,B}^T\boldsymbol{\lambda}_h^n, \\
        \begin{bmatrix}
             \mathcal A_{E,II}    &    \mathcal B_{E,I}^T    \\
             \mathcal B_{E,I}    &    -\mathcal C_E
        \end{bmatrix}\begin{bmatrix}
            \mathcal U_{E,h,I}^n    \\
            \mathcal P_{E,h}^n    \\
        \end{bmatrix}&=\begin{bmatrix}
            \mathcal F_{E,I}^n    \\
            \mathcal O
        \end{bmatrix} - \begin{bmatrix}
            \mathcal A_{E,IB}^T    \\    \mathcal B_{E,B}^T
        \end{bmatrix}\mathcal U_{E,h,B}^n  - \mathcal H_{E,B}^T\boldsymbol{\lambda}_h^n
    \end{aligned}\] for $\mathcal U_{P,h,I}^n$ and $\mathcal U_{E,h,I}^n$. The existence and uniqueness of solutions of the above two subsystems can be easily proved by Proposition 3.3.1 in \cite{boffi_mixed_2013}.
    \end{proof}

\begin{remark}
    The proof of Theorem \ref{thm:ex-uni} shows that $\mathcal B_{\mathcal D,I}$ need not be surjective. Corresponding to the choices of finite element spaces, there is no need to satisfy the inf-sup condition, which means that the lowest-order finite element spaces can be chosen without vanishing the existence and uniqueness of solutions.
\end{remark}

By eliminating $\mathcal H_{P,B}$, $\mathcal H_{E,B}$ in the third block row, we can obtain the Schur complement system of \eqref{eqn:reduced-system}: \begin{equation}\label{eqn:feti-system}
    \mathcal K_S\boldsymbol{\lambda}_h^n=\mathcal F_S^n,
\end{equation}
where\[\begin{aligned}
\mathcal K_S&=\mathcal H_{P,B}\mathcal S_P^{-1}\mathcal H_{P,B}^T + \mathcal H_{E,B}\mathcal S_E^{-1}\mathcal H_{E,B}^T \\
\mathcal F_S^n&=\mathcal H_{P,B}\mathcal S_P^{-1}\mathcal H_{P,B}^T\tilde{\mathcal F}_{P,B} + \mathcal H_{E,B}\mathcal S_E^{-1}\tilde{\mathcal F}^n_{E,B}.
\end{aligned}\]
Since $\mathcal K_S,\ \mathcal F_S$ can neither be expressed explicitly, we consider a PCG method for solving \eqref{eqn:feti-system}. In this paper, we consider the FETI Dirichlet preconditioner: \begin{equation}\label{eqn:feti_dirichlet_preconditioner}
    \mathcal M^{-1}_S := \mathcal H_{P,B}\mathcal S_P\mathcal H_{P,B}^T + \mathcal H_{E,B}\mathcal S_E\mathcal H_{E,B}^T.
\end{equation}

The FETI system \eqref{eqn:feti-system} is considered to be solved using PCG method with left conditioner $\mathcal M_S$ (see Algorithm \ref{alg:PCG-FETI}). However, since the matrix $\mathcal K_S$ is too complicated to be explicitly presented, that is, calculating $\mathcal K_s\mathbf p_k$ for each $k$ in Algorithm \ref{alg:PCG-FETI} involves solving a Schur complement system with coefficient matrices $\mathcal S_\mathcal{D}$, which needs implementing PCG method in each PCG step resulting in a multiplicative superposition of PCG iterations.

\begin{algorithm}[!htbp]
    \caption{PCG algorithm for solving \eqref{eqn:feti-system}}
    \label{alg:PCG-FETI}
    \begin{algorithmic}[1]
        \STATE Let $\boldsymbol{\lambda}_h^0=\mathbf 0$
        \FOR{$n=1,2,\cdots,N$}
        \STATE Let $\boldsymbol{\lambda}_h^{(n,0)}:=\boldsymbol{\lambda}_h^{n-1},\ \mathbf r_0:=\mathcal F_S^n-\mathcal K_S\boldsymbol{\lambda}_h^{(n,0)},\ \mathbf z_0:=\mathcal M^{-1}_S\mathbf r_0,\ \mathbf p_0:=\mathbf z_0$
        \FOR{$k=0,1,2,\cdots$ until convergence $\boldsymbol{\lambda}_{h}^{(n,k)}\to\boldsymbol{\lambda}_{h}^n$}
        \STATE $\alpha_{k}=\left(\mathbf r_{k},\mathbf z_{k}\right)/\left(\mathbf p_{k},\mathcal K_S\mathbf p_{k}\right)$
        \STATE $\boldsymbol{\lambda}_h^{(n,k+1)}:=\boldsymbol{\lambda}_h^{(n,k)}+\alpha_k\mathbf p_k$
        \STATE $\mathbf r_{k+1}=\mathbf r_k-\alpha_k\mathcal K_S\mathbf p_k$
        \STATE $\mathbf z_{k+1}=\mathcal M_S^{-1}\mathbf z_k$
        \STATE $\beta_k=\left(\mathbf r_{k+1},\mathbf z_{k+1}\right) / \left(\mathbf r_k,\mathbf z_k\right)$
        \STATE $\mathbf p_{k+1}=\mathbf z_{k+1}+\beta_k\mathbf p_k$
        \ENDFOR
        \ENDFOR
        \ENSURE $\boldsymbol{\lambda}_{h}^{n}\ (n=1,2,\cdots,N)$
    \end{algorithmic}
\end{algorithm}

This theoretical fact is far from satisfactory for computation. A generalized implementation proposed in \cref{sect:a_generailzed_implementation} will naturally avoid the multiplicative superposition of PCG iterations, that is, we do not distinguish between interior (including boundary) and interface points, and correspondingly we follow the idea for constructing the FETI Dirichlet preconditioner \eqref{eqn:feti_dirichlet_preconditioner} to defined an analogous preconditioner for solving the generalized system.

\subsection{A Generalized Implementation} \label{sect:a_generailzed_implementation}

As the last paragraph in \cref{sect:the_feti_algorithm_and_preconditioner} describes, to solve \eqref{eqn:feti-system} will unavoidably treat several Schur complement systems, where the PCG algorithm is needed. Here we consider not distinguishing interior and interface points, that is, the following Schur complement system: \begin{equation}\label{eqn:sim-feti}
    \mathcal K_A\boldsymbol{\lambda}_h^n = \mathcal{F}_A,
\end{equation}
where \[\begin{aligned}
\mathcal K_A&=\begin{bmatrix}
    \mathcal H_{P}^*    &    \mathcal O
\end{bmatrix}\begin{bmatrix}
    \mathcal A_P^*    &    (\mathcal B_P^*)^T \\
        \mathcal B_P^*    &    -\mathcal C_P^* 
\end{bmatrix}^{-1}\begin{bmatrix}
    \mathcal (H_{P}^*)^T    \\    \mathcal O
\end{bmatrix} + \begin{bmatrix}
    \mathcal H_{E}    &    \mathcal O
\end{bmatrix}\begin{bmatrix}
    \mathcal A_E    &    \mathcal B_E^T \\
        \mathcal B_E    &    -\mathcal C_E 
\end{bmatrix}^{-1}\begin{bmatrix}
    \mathcal H_{E}^T    \\    \mathcal O
\end{bmatrix}, \\
\mathcal F_A&=\begin{bmatrix}
    \mathcal H_{P}    &    \mathcal O
\end{bmatrix}\begin{bmatrix}
    \mathcal A_P^*    &    (\mathcal B_P^*)^T \\
        \mathcal B_P^*    &    -\mathcal C_P^* 
\end{bmatrix}^{-1}\begin{bmatrix}
    \mathcal F_{P,*}^n    \\
        \mathcal Z^n      
\end{bmatrix} + \begin{bmatrix}
    \mathcal H_{E}    &    \mathcal O
\end{bmatrix}\begin{bmatrix}
    \mathcal A_E    &    \mathcal B_E^T \\
        \mathcal B_E    &    -\mathcal C_E 
\end{bmatrix}^{-1}\begin{bmatrix}
    \mathcal F_E^n    \\
        \mathcal O
\end{bmatrix}.
\end{aligned}\]
It can be verified by Sylvester's law that $\mathcal K_A$ is also symmetric and positive definite. Analogously, we consider the following preconditioner: \begin{equation}
    \mathcal M^{-1}_A := \begin{bmatrix}
    \mathcal H_{P}^*    &    \mathcal O
\end{bmatrix}\begin{bmatrix}
    \mathcal A_P^*    &    (\mathcal B_P^*)^T \\
        \mathcal B_P^*    &    -\mathcal C_P^* 
\end{bmatrix}\begin{bmatrix}
    \mathcal (H_{P}^*)^T    \\    \mathcal O
\end{bmatrix} + \begin{bmatrix}
    \mathcal H_{E}    &    \mathcal O
\end{bmatrix}\begin{bmatrix}
    \mathcal A_E    &    \mathcal B_E^T \\
        \mathcal B_E    &    -\mathcal C_E 
\end{bmatrix}\begin{bmatrix}
    \mathcal H_{E}^T    \\    \mathcal O
\end{bmatrix}.
\end{equation}

Analogously, the generalized FETI system \eqref{eqn:sim-feti} is considered to be solved using PCG method with left conditioner $\mathcal M_A$ (see Algorithm \ref{alg:PCG-FETI-SIMPLE}). In actuality, $\mathcal M_A^{-1}$ can be written in \[
    \mathcal M_A^{-1}=\mathcal H_{P,B}\mathcal A_{P,BB}\mathcal H_{P,B}^T+\mathcal H_{E,B}\mathcal A_{E,BB}\mathcal H_{E,B}^T.
\]

Systems \eqref{eqn:sim-feti} has the same solution as \eqref{eqn:reordered-system}. The difference between $\mathcal M_A$ and $\mathcal M_S$ is that $\mathcal M_S$ has the Schur complement term \[
    \mathcal H_{P,B}\mathcal A_{P,IB}\mathcal A_{P,II}^{-1}\mathcal A_{P,IB}^T\mathcal H_{P,B}^T+\mathcal H_{E,B}\mathcal A_{E,IB}\mathcal A_{E,II}^{-1}\mathcal A_{E,IB}^T\mathcal H_{E,B}^T,
\] which implies that $\mathcal M_A$ has lower cost of computation resource, that is, if we consider solving each subproblem in each subdomain using Schur complement method, it is necessary to adopt PCG algorithm in each iteration. However, the double PCG iterations implies multiplicative maximum number of iterations, which means that the scale of stiffness equations must not be too large in order to reduce the computational costs. 

\begin{algorithm}[!htbp]
    \caption{PCG algorithm for solving \eqref{eqn:sim-feti}}
    \label{alg:PCG-FETI-SIMPLE}
    \begin{algorithmic}[1]
        \STATE Let $\boldsymbol{\lambda}_h^0=\mathbf 0$
        \FOR{$n=1,2,\cdots,N$}
        \STATE Let $\boldsymbol{\lambda}_h^{(n,0)}:=\boldsymbol{\lambda}_h^{n-1},\ \mathbf r_0:=\mathcal F_A^n-\mathcal K_A\boldsymbol{\lambda}_h^{(n,0)},\ \mathbf z_0:=\mathcal M^{-1}_A\mathbf r_0,\ \mathbf p_0:=\mathbf z_0$
        \FOR{$k=0,1,2,\cdots$ until convergence $\boldsymbol{\lambda}_{h}^{(n,k)}\to\boldsymbol{\lambda}_{h}^n$}
        \STATE $\alpha_{k}=\left(\mathbf r_{k},\mathbf z_{k}\right)/\left(\mathbf p_{k},\mathcal K_A\mathbf p_{k}\right)$
        \STATE $\boldsymbol{\lambda}_h^{(n,k+1)}:=\boldsymbol{\lambda}_h^{(n,k)}+\alpha_k\mathbf p_k$
        \STATE $\mathbf r_{k+1}=\mathbf r_k-\alpha_k\mathcal K_S\mathbf p_k$
        \STATE $\mathbf z_{k+1}=\mathcal M_A^{-1}\mathbf z_k$
        \STATE $\beta_k=\left(\mathbf r_{k+1},\mathbf z_{k+1}\right) / \left(\mathbf r_k,\mathbf z_k\right)$
        \STATE $\mathbf p_{k+1}=\mathbf z_{k+1}+\beta_k\mathbf p_k$
        \ENDFOR
        \ENDFOR
        \ENSURE $\boldsymbol{\lambda}_{h}^{n}\ (n=1,2,\cdots,N)$
    \end{algorithmic}
\end{algorithm}

\section{Numerical Results}
\label{sect:computational_resutls}

In this section, we consider several numerical tests cases for validating the computation efficiency of the above algorithms. First we adopt a model that has exact solutions in the poroelasticity subdomain, and from this the exact solutions in the elasticity domain can be manufactured. In this model we consider the two settings of finite element discretization: the ($\boldsymbol{P}_2$-$P_1$-$P_1$-$P_1$, $\boldsymbol{P}_2$-$P_1$) finite element space and the lowest-order ($\boldsymbol{P}_1$-$P_1$-$P_1$-$P_1$, $\boldsymbol{P}_1$-$P_1$) finite element space, where the convergence order and the parametric robustness to the Poisson's ratio. Then the Barry-Mercer's problem is focused as a benchmark test problem.

\subsection{Test Case 1: Convergence Order Validation}\label{ex:model_validation}
    Consider the unit square domain $\Omega=[0,1]\times [0,1]$, where $\Omega^P=[0,1]\times[0,1/2]$ and $\Omega^E=[0,1]\times[1/2,1]$. Time interval and step are set in $T=10^{-2}$ and $\tau=10^{-4}$. We set the exact solution of $\mathbf u$, $p$ in $\Omega^P$ in the following: \[
        \mathbf u_P=\begin{bmatrix}
            \sin(2\pi x)\sin(2\pi y) \\
            \sin(2\pi x)\sin(2\pi y)
        \end{bmatrix},\quad p=\sin(\pi x)\sin(\pi y).
    \]
    To guarantee that the transmission condition on $\varGamma$ holds, we consider the following manufactured solutions: \[
        \mathbf u_E=\mathbf u_P+\begin{bmatrix}
            0    \\
            -\frac{\alpha p(y-1/2)}{\lambda_P+2\mu_P}
        \end{bmatrix}.
    \]
    In this test case we set the ($\boldsymbol{P}_2$-$P_1$-$P_1$-$P_1$, $\boldsymbol{P}_2$-$P_1$) finite element space for the spatial discretization of the model. The parametric settings are listed in \cref{tab:para_in_ex1}, where the Poisson's ratio $\nu_{\mathcal D}$ are multiply set for validating that the model can overcome the ``locking" of both displacement and pressure. 

    \begin{table}[!htbp]
        \small
        \centering
        \caption{Parameters in \cref{ex:model_validation}}
        \label{tab:para_in_ex1}
        \begin{tabularx}{\textwidth}{XXX}
            \toprule[1pt]
            Parameter     &      Description \& Value    \\
            \midrule
            $E_\mathcal{D}$    &    Young's modulus    &    $E_P=E_E=10^4$    \\
            $\nu_\mathcal{D}$    &    Poisson's ratio    &    $\nu_P=\nu_E=0.2, 0.49, 0.499, 0.4999$    \\
            $\lambda_\mathcal{D},\ \mu_\mathcal{D}$    &    Lam\'e constants     &    $\lambda_\mathcal{D}=\frac{E_\mathcal{D}\nu_\mathcal{D}}{(1+\nu_{\mathcal{D}})(1-2\nu_{\mathcal D})},\ \mu_\mathcal{D}=\frac{E_\mathcal{D}}{1+\nu_{\mathcal{D}}}$   \\
            $\alpha$    &    Biot-Willis constant    &    1    \\
            $c_0$    &    Fluid specific storage coefficient    &    0.1    \\
            $\boldsymbol{K}$    &    Skeleton permeability tensor    &    1$\mathbf I_2$    \\
            $\mu_f$    &    Fluid viscosity    &    1    \\
            \bottomrule[1pt]
          \end{tabularx}
    \end{table} 
    
    We calculated the model problem in respectively $8\times 8$, $16\times 16$, $24\times 24$ and $32\times 32$ uniform triangulation meshes, and $L^\infty(L^2)$ errors and convergence orders are shown in \cref{tab:error_order_ex1}. To present the results, we plot the computational results of the displacement $\mathbf u$ and pressure $p$ at $\nu_\mathcal{D}=0.2$, $h=1/16$ and $t=0.001$ with heatmaps (see \cref{fig:ex1_result}). The results shows the convergence of the numerical model, and the superconvergence occurs for both displacement $\mathbf{u}$ and pressure $p$, that is, the higher spatial convergence error order than the interpolation error orders of the corresponding finite element spaces.
    
    \begin{table}[!htbp]
        \small
        \centering
        \caption{ Errors in $L^\infty(L^2)$-norm  and convergence order for test case 1}
        \label{tab:error_order_ex1}
        \begin{tabularx}{\textwidth}{XXXXXX}
            \toprule[1pt]
            $\nu$    &    $h$	 &	$\|\mathbf u_{h}-\mathbf u\|$	&	Order	&	$\|p_h-p\|$	&	Order	\\
            \midrule
            \multirow{4}{*}{0.2} & $1/8$	&	$3.4464\times10^{-5}$	&	$-$	&	$8.3059\times10^{-5}$	&	$-$	\\
            &    $1/16$	&	$4.9157\times10^{-6}$	&	2.8096	&	$1.7016\times10^{-5}$	&	2.2872	\\
            &    $1/24$	&	$1.5306\times10^{-6}$	&	2.8775	&	$7.7335\times10^{-6}$	&	1.9322	\\
            &    $1/32$	&	$6.8132\times 10^{-6}$	&	2.8135	&	$4.8641\times 10^{-6}$	&	1.6297	\\
            \midrule
            \multirow{4}{*}{0.49} & $1/8$	&	$4.4612\times10^{-4}$	&	$-$	&	$5.8174\times10^{-5}$	&	$-$	\\
            &    $1/16$	&	$3.0153\times10^{-5}$	&	3.8870  &	$1.3567\times10^{-5}$	&	2.1003	\\
            &    $1/24$	&	$6.0305\times10^{-5}$	&	3.9694	&	$6.6025\times10^{-6}$	&	1.7762	\\
            &    $1/32$	&	$1.9107\times 10^{-6}$	&	3.9953	&	$3.3601\times 10^{-6}$	&	2.1071	\\
            \midrule
            \multirow{4}{*}{0.499} & $1/8$	&	$4.8806\times10^{-3}$	&	$-$	&	$5.9630\times10^{-5}$	&	$-$	\\
            &    $1/16$	&	$3.1105\times10^{-4}$	&	3.9719  &	$1.3508\times10^{-5}$	&	2.1422	\\
            &    $1/24$	&	$5.9523\times10^{-5}$	&	4.0782	&	$6.6926\times10^{-6}$	&	1.7693	\\
            &    $1/32$	&	$1.8070\times 10^{-5}$	&	4.1439	&	$3.2951\times 10^{-6}$	&	2.4106	\\
            \midrule
            \multirow{4}{*}{0.4999} & $1/8$	&	$4.7204\times10^{-2}$	&	$-$	&	$1.1970\times10^{-4}$	&	$-$	\\
            &    $1/16$	&	$3.0931\times10^{-3}$	&	3.9318  &	$1.4237\times10^{-5}$	&	3.0717	\\
            &    $1/24$	&	$5.9446\times10^{-4}$	&	4.0676	&	$5.8680\times10^{-6}$	&	2.1860	\\
            &    $1/32$	&	$1.8671\times 10^{-4}$	&	4.0256	&	$3.9846\times 10^{-6}$	&	1.3455	\\
            \bottomrule[1pt]
        \end{tabularx}
    \end{table}
    
\begin{figure}[!htbp]
    \centering
	\begin{subfigure}{0.3\linewidth}
		\includegraphics[width=\linewidth]{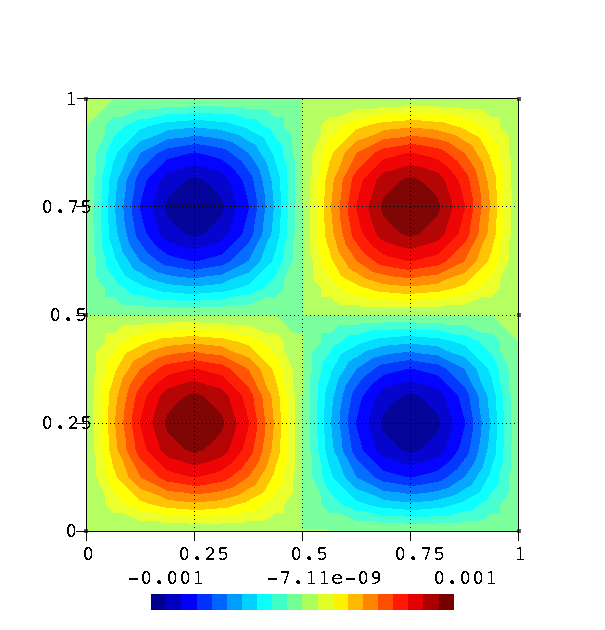}
		\caption{$u_{h,1}$}
	\end{subfigure}
	\begin{subfigure}{0.3\linewidth}
		\includegraphics[width=\linewidth]{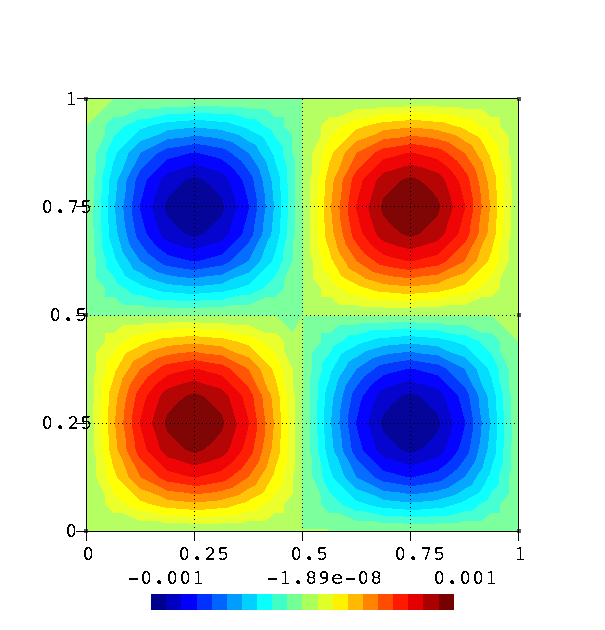}
		\caption{$u_{h,2}$}
	\end{subfigure}
	\begin{subfigure}{0.3\linewidth}
		\includegraphics[width=\linewidth]{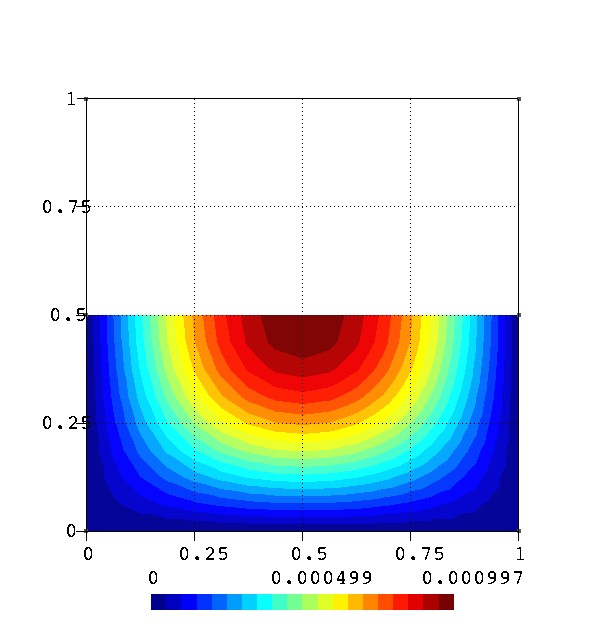}
		\caption{$p_h$}
	\end{subfigure}
	\caption{Plots of solutions at $\nu=0.02$, $h=1/16$ and $t=0.001$.}
	\label{fig:ex1_result}
\end{figure}

The following \cref{fig:ex1_error} presents the errors of displacement $\mathbf u$ and pressure $p$ respectively when adopting different Poisson's ratios $\nu_\mathcal{D}$. From this we find that the error of $\mathbf u$ increases and the error of $p$ is almost unchanged. However, the error or $p$ when $h=1/8$ and $\nu_\mathcal{D}=0.4999$ is much larger than other test positions. This exception can be explained by the influence of relatively large error of displacement $\mathbf{u}$ (greater than $10^{-2}$). We can from the above observations conclude that the model is robust to Poisson's ratio and can overcome the ``locking" of both displacement and pressure.

\begin{figure}[!htbp]
    \centering
	\begin{subfigure}{0.4\linewidth}
		\includegraphics[width=\linewidth]{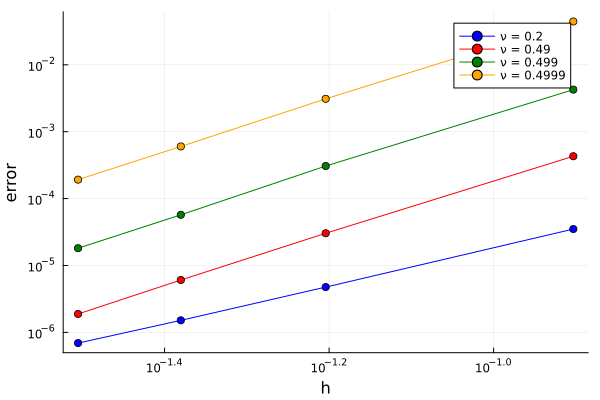}
		\caption{$\left\|\mathbf u_h-\mathbf u\right\|$}
	\end{subfigure}
	\begin{subfigure}{0.4\linewidth}
		\includegraphics[width=\linewidth]{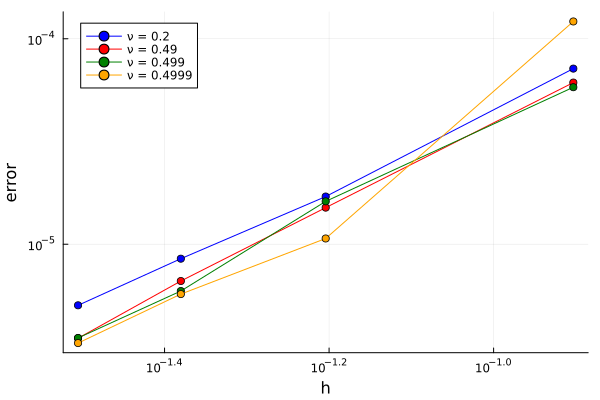}
		\caption{$\left\|p_h-p\right\|$}
	\end{subfigure}
 \caption{Error plots in $L^\infty(L^2)$-norm in test case 1}
	\label{fig:ex1_error}
\end{figure}

\subsection{Test Case 2: Lowest-order Finite Element Space}
    In this numerical test we consider adopting the lowest-order finite element space, that is, the ($\boldsymbol{P}_1$-$P_1$-$P_1$-$P_1$, $\boldsymbol{P_1}$-$P_1$) finite element space for spatial discretization, and the domain, exact solutions and parameter settings are all the same as test case 1 (see). We calculated again the model problem in respectively $8\times 8$, $16\times 16$, $24\times 24$ and $32\times 32$ uniform triangulation meshes, and $L^\infty(L^2)$ errors and convergence orders are shown in \cref{tab:error_order_ex2}. We can find that the model converges like the above settings of ($\boldsymbol{P}_2$-$P_1$-$P_1$-$P_1$, $\boldsymbol{P}_2$-$P_1$) finite element space, and the superconvergence phenomenon that has shown in test case 1 still exists in this lowest-order finite element case.

    \begin{table}[!htbp]
        \small
        \centering
        \caption{ Errors in $L^\infty(L^2)$-norm  and convergence order for test case 2}
        \label{tab:error_order_ex2}
        \begin{tabularx}{\textwidth}{XXXXXX}
            \toprule[1pt]
            $\nu$    &    $h$	 &	$\|\mathbf u_{h}-\mathbf u\|$	&	Order	&	$\|p_h-p\|$	&	Order	\\
            \midrule
            \multirow{4}{*}{0.2} & $1/8$	&	$5.4843\times10^{-4}$	&	$-$	&	$1.9126\times10^{-3}$	&	$-$	\\
            &    $1/16$	&	$1.7852\times10^{-4}$	&	1.6192	&	$6.4009\times10^{-4}$	&	1.5792	\\
            &    $1/24$	&	$7.6087\times10^{-5}$	&	2.1033	&	$2.9777\times10^{-4}$	&	1.8874	\\
            &    $1/32$	&	$4.1245\times 10^{-4}$	&	2.1284	&	$1.4564\times 10^{-4}$	&	2.4859	\\
            \midrule
            \multirow{4}{*}{0.49} & $1/8$	&	$4.8635\times10^{-4}$	&	$-$	&	$1.9548\times10^{-4}$	&	$-$	\\
            &    $1/16$	&	$1.3026\times10^{-4}$	&	1.9005  &	$5.5812\times10^{-5}$	&	1.8084	\\
            &    $1/24$	&	$5.4334\times10^{-5}$	&	2.1565	&	$2.0711\times10^{-5}$	&	2.4448	\\
            &    $1/32$	&	$2.9772\times 10^{-5}$	&	2.0911	&	$1.1630\times 10^{-5}$	&	2.0060	\\
            \midrule
            \multirow{4}{*}{0.499} & $1/8$	&	$5.7018\times10^{-4}$	&	$-$	&	$1.5800\times10^{-4}$	&	$-$	\\
            &    $1/16$	&	$1.4273\times10^{-4}$	&	1.9981  &	$3.1721\times10^{-5}$	&	2.3164	\\
            &    $1/24$	&	$5.7814\times10^{-5}$	&	2.2289	&	$1.2862\times10^{-5}$	&	2.2262	\\
            &    $1/32$	&	$3.2537\times 10^{-5}$	&	1.9981	&	$6.2831\times 10^{-6}$	&	2.4905	\\
            \midrule
            \multirow{4}{*}{0.4999} & $1/8$	&	$1.6353\times10^{-3}$	&	$-$	&	$1.4759\times10^{-4}$	&	$-$	\\
            &    $1/16$	&	$1.6725\times10^{-4}$	&	3.2894  &	$3.49312\times10^{-5}$	&	2.0790	\\
            &    $1/24$	&	$6.1010\times10^{-5}$	&	2.4872	&	$1.2877\times10^{-5}$	&	2.4612	\\
            &    $1/32$	&	$3.2578\times 10^{-5}$	&	2.1808	&	$6.2396\times 10^{-6}$	&	2.5185	\\
            \bottomrule[1pt]
        \end{tabularx}
    \end{table}

The following \cref{fig:ex2_error} presents the errors of displacement $\mathbf u$ and pressure $p$ respectively when adopting different Poisson's ratios $\nu_\mathcal{D}$. We can see that the errors of displacement $\mathbf{u}$ and pressure $p$ have the minimal change when $\nu_\mathcal{D}$ is gradually close to $0.5$. However, the error of $\mathbf{u}$ does not show an increasing phenomenon like test case 1, and the error of $p$ at $\nu_\mathcal{D}=0.2$ is much larger than other settings of $\nu_\mathcal{D}$. The results totally shows that the model can overcome the ``locking" of both displacement and pressure, and has stronger robustness when adopting the lowest-order finite element discretization than what test case 1 has shown.

    \begin{figure}[!htbp]
    \centering
	\begin{subfigure}{0.4\linewidth}
		\includegraphics[width=\linewidth]{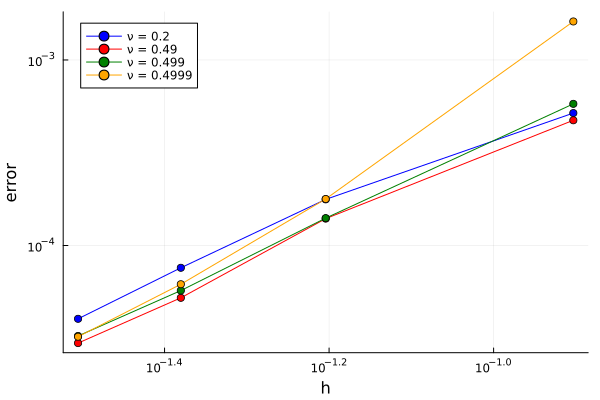}
		\caption{$\left\|\mathbf u_h-\mathbf u\right\|$}
	\end{subfigure}
	\begin{subfigure}{0.4\linewidth}
		\includegraphics[width=\linewidth]{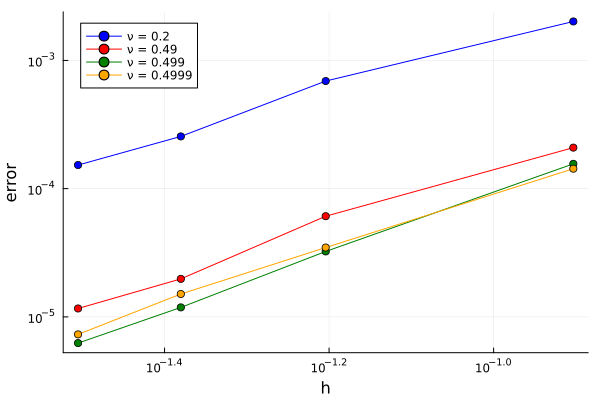}
		\caption{$\left\|p_h-p\right\|$}
	\end{subfigure}
	\label{fig:ex2_error}
 \caption{Error plots in $L^\infty(L^2)$-norm in test case 2}
\end{figure}

\subsection{Test Case 3: Barry–Mercer’s Problem}
    In this numerical test we consider the Barry–Mercer’s problem as the benchmark test problem. We also consider the unique square domain defined in \cref{ex:model_validation} (see \cref{fig:fig-barry-mercer}), and $T=1$. Source terms and boundary conditions are set in the following. 
    \begin{alignat*}{3}
        \mathbf f&\equiv\mathbf 0,	\\
        z&\equiv 0,	\\
        p&=0&&\quad\text{on }\partial\Omega^P\cap(\varGamma_1\cup\varGamma_3), \\
        \mathbf w(p)\cdot\mathbf n_{P\to E}&=0&&\quad\text{on }\varGamma, \\
        p&=p_2&&\quad\text{on }\varGamma_2,\ \text{where } p_2 = \begin{cases}
            \sin t&	\text{if }0.2\leq x\leq 0.8 \\
            0&\text{otherwise}
        \end{cases}\\
        u_1&=0&&\quad\text{on }\varGamma_1\cup\varGamma_3,	\\
        u_2&=0&&\quad\text{on }\varGamma_2\cup\varGamma_4,	\\
        \tilde{\boldsymbol{\sigma}}_P(\mathbf u_P,\xi_P)\cdot\mathbf n_P&=(0,\alpha p_2)^T&&\quad\text{on }\varGamma_2, \\
        \mathbf u_0(\mathbf x)&=\mathbf 0, \\
        p_0(\mathbf x)&=0,
    \end{alignat*}

    \begin{figure}[!htbp]
        \centering
        \begin{subfigure}{0.4\linewidth}
		\begin{tikzpicture}
            \draw (-3, -3) -- (3, -3);
            \draw (-3, 0) -- (3, 0);
            \draw (-3, 3) -- (3, 3);
            \draw (-3, -3) -- (-3, 3);
            \draw (3, -3) -- (3, 3);
            \draw (0, -0.3) node {$\varGamma$};
            \draw (0, 3.3)  node {$\varGamma_4$};
            \draw (-3.3, 0) node {$\varGamma_3$};
            \draw (3.3, 0) node {$\varGamma_1$};
            \draw (0, -3.3) node {$\varGamma_2$};
            \draw (0, -1.5) node  {$\Omega^P$};
            \draw (0, 1.5) node {$\Omega^E$};
        \end{tikzpicture}
        \caption{Domain and its partition}
        \label{fig:fig-barry-mercer}
	\end{subfigure}
     \begin{subfigure}{0.4\linewidth}
        \includegraphics[width=\linewidth]{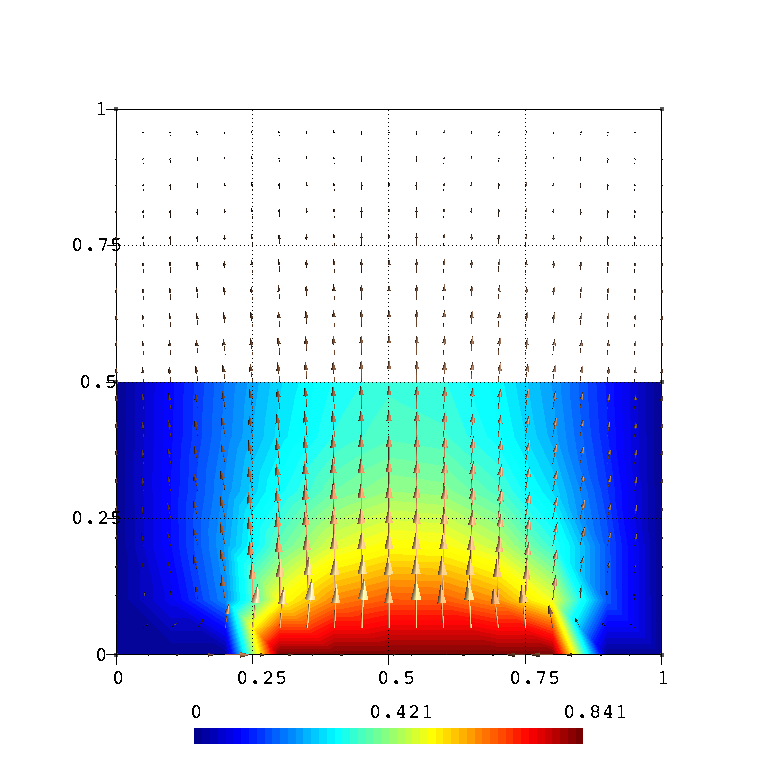}
        \caption{Numerical pressure $p$ (heatmap) and displacement $\mathbf u$ (arrows)}
        \label{fig:barrymercer}
     \end{subfigure}
    \caption{Domain and numerical result for Barry-Mercer's problem}
    \end{figure}
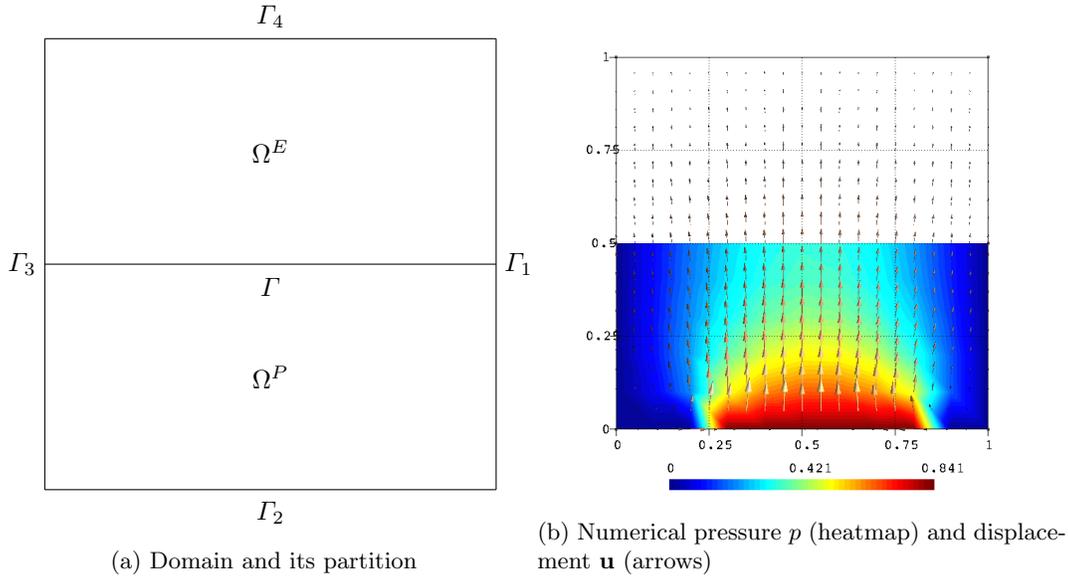
Note that Barry–Mercer’s problem has a unique solution that is given by a series (see \cite{phillips_overcoming_2009}), and the manufactured solutions in the elasticity domain can be found using the similar technique in \cref{ex:model_validation}. The computational result under the above settings is shown in \cref{fig:barrymercer}, where the computed pressure in $[0,1]\times[0,1/2]$ is shown by a heat map and the computed displacement in $[0,1]\times[0,1]$ is displayed by arrows. The results shows that \cref{alg:PCG-FETI} and \cref{alg:PCG-FETI-SIMPLE} performs well in computing the Barry-Mercer's problem and does not produce any oscillation in computed pressure.

\section{Conclusion}

In this paper, we proposed an time-efficient algorithm for solving a coupled poroelasticity and elasticity model with multiphysics finite element approximation in a Lagrange multiplier framework. The fully-discretized problem was reconstructed into a saddle point form, and then the specifically designed multiphysics FETI method was adopted for iterative solving. Numerical tests showed that the algorithm possess the time efficiency and high accuracy in computational results, and the ``locking" phenomenon was overcome according to the results under different Poisson's ratios. And the domain decomposition method revealed the strong potential for parallelization, and strategies and techniques for constructing the algorithm can be applied to the more complicated multiphysics domain decomposition problems.

\bibliographystyle{siam}
\bibliography{references.bib}

\begin{thebibliography}{10}

\bibitem{abdi_modeling_2021}
{\sc J.~Abdi, F.~Hadavimoghaddam, M.~Hadipoor, and A.~Hemmati-Sarapardeh}, {\em
  Modeling of {CO$_2$} adsorption capacity by porous metal organic frameworks
  using advanced decision tree-based models}, Scientific reports, 11 (2021),
  p.~24468.

\bibitem{antonietti2007schwarz}
{\sc P.~F. Antonietti and B.~Ayuso}, {\em Schwarz domain decomposition
  preconditioners for discontinuous galerkin approximationsof elliptic
  problems: non-overlapping case}, ESAIM: Mathematical Modelling and Numerical
  Analysis, 41 (2007), pp.~21--54.

\bibitem{biot_general_1941}
{\sc M.~A. Biot}, {\em General theory of three-dimensional consolidation},
  Journal of applied physics, 12 (1941), pp.~155--164.

\bibitem{boffi_mixed_2013}
{\sc D.~Boffi, F.~Brezzi, M.~Fortin, et~al.}, {\em Mixed finite element methods
  and applications}, vol.~44, Springer, 2013.

\bibitem{chen_physical_2023}
{\sc H.~Chen and Z.~Ge}, {\em Physical information neural networks for {2D} and
  {3D} nonlinear biot model and simulation on the pressure of brain}, Journal
  of Computational Physics, 490 (2023), p.~112309.

\bibitem{chu2025block}
{\sc H.~Chu, L.~F. Pavarino, and S.~Zampini}, {\em Block bddc/feti-dp
  preconditioners for three-field mixed finite element discretizations of
  biot's consolidation model}, arXiv preprint arXiv:2504.04859,  (2025).

\bibitem{coussy_poromechanics_2004}
{\sc O.~Coussy}, {\em Poromechanics}, John Wiley \& Sons, 2004.

\bibitem{dai_co2_2017}
{\sc Z.~Dai, H.~Viswanathan, T.~Xiao, R.~Middleton, F.~Pan, W.~Ampomah,
  C.~Yang, Y.~Zhou, W.~Jia, S.-Y. Lee, et~al.}, {\em {CO$_2$} sequestration and
  enhanced oil recovery at depleted oil/gas reservoirs}, Energy Procedia, 114
  (2017), pp.~6957--6967.

\bibitem{feng_analysis_2018}
{\sc X.~Feng, Z.~Ge, and Y.~Li}, {\em Analysis of a multiphysics finite element
  method for a poroelasticity model}, IMA Journal of Numerical Analysis, 38
  (2018), pp.~330--359.

\bibitem{feng_phase-field_2023}
{\sc Y.~Feng and A.~Firoozabadi}, {\em Phase-field simulation of hydraulic
  fracturing by {CO$_2$} and water with consideration of thermoporoelasticity},
  Rock Mechanics and Rock Engineering, 56 (2023), pp.~7333--7355.

\bibitem{girault_domain_2011}
{\sc V.~Girault, G.~Pencheva, M.~F. Wheeler, and T.~Wildey}, {\em Domain
  decomposition for poroelasticity and elasticity with {DG} jumps and mortars},
  Mathematical Models and Methods in Applied Sciences, 21 (2011), pp.~169--213.

\bibitem{hochbruck2010exponential}
{\sc M.~Hochbruck and A.~Ostermann}, {\em Exponential integrators}, Acta
  Numerica, 19 (2010), pp.~209--286.

\bibitem{koko2010uzawa}
{\sc J.~Koko and T.~Sassi}, {\em An uzawa domain decomposition method for
  stokes problem}, in Domain Decomposition Methods in Science and Engineering
  XIX, Springer, 2010, pp.~383--390.

\bibitem{lee2012analysis}
{\sc K.~Lee, M.~Tak, and T.~Park}, {\em The analysis of porous media using the
  mixed finite element method and the feti method}, Civil-Comp Proceedings, 100
  (2012), pp.~1--11.

\bibitem{lee2010mixed}
{\sc K.-J. Lee, M.-H. Tak, and T.-H. Park}, {\em The mixed finite element
  analysis for saturated porous media using feti method}, Journal of the
  Computational Structural Engineering Institute of Korea, 23 (2010),
  pp.~693--702.

\bibitem{lee2022feti}
{\sc P.~Lee}, {\em Feti-dp algorithms for 2d biot model with discontinuous
  galerkin discretization}, in International Conference on Domain Decomposition
  Methods, Springer, 2022, pp.~351--358.

\bibitem{magoules2004non}
{\sc F.~Magoules, P.~Iv{\'a}nyi, and B.~H. Topping}, {\em Non-overlapping
  schwarz methods with optimized transmission conditions for the helmholtz
  equation}, Computer Methods in Applied Mechanics and Engineering, 193 (2004),
  pp.~4797--4818.

\bibitem{magoules2005optimal}
{\sc F.~MAGOUL{\`E}S and R.~Putanowicz}, {\em Optimal convergence of
  non-overlapping schwarz methods for the helmholtz equation}, Journal of
  Computational Acoustics, 13 (2005), pp.~525--545.

\bibitem{mathew_domain_2008}
{\sc T.~P. Mathew}, {\em Domain decomposition methods for the numerical
  solution of partial differential equations}, Springer, 2008.

\bibitem{ning2018uzawa}
{\sc P.~Ning, Z.-Q. Feng, J.~A.~R. Quintero, Y.-J. Zhou, and L.~Peng}, {\em
  Uzawa algorithm to solve elastic and elastic--plastic fretting wear problems
  within the bipotential framework}, Computational Mechanics, 62 (2018),
  pp.~1327--1341.

\bibitem{phillips_coupling_2007}
{\sc P.~J. Phillips and M.~F. Wheeler}, {\em A coupling of mixed and continuous
  {Galerkin} finite element methods for poroelasticity {I}: the continuous in
  time case}, Computational Geosciences, 11 (2007), pp.~131--144.

\bibitem{phillips_coupling_2007-1}
\leavevmode\vrule height 2pt depth -1.6pt width 23pt, {\em A coupling of mixed
  and continuous {Galerkin} finite element methods for poroelasticity {II}: the
  discrete-in-time case}, Computational Geosciences, 11 (2007), pp.~145--158.

\bibitem{phillips_overcoming_2009}
\leavevmode\vrule height 2pt depth -1.6pt width 23pt, {\em Overcoming the
  problem of locking in linear elasticity and poroelasticity: an heuristic
  approach}, Computational Geosciences, 13 (2009), pp.~5--12.

\bibitem{terzaghi_theoretical_1943}
{\sc K.~Terzaghi}, {\em Theoretical soil mechanics}, John Wiley \& Sons, Inc.,
  1943.

\bibitem{yankova_study_2020}
{\sc G.~S. Yan’kova, A.~A. Cherevko, A.~K. Khe, O.~B. Bogomyakova, and A.~A.
  Tulupov}, {\em Study of hydrocephalus using poroelastic models}, Journal of
  Applied Mechanics and Technical Physics, 61 (2020), pp.~14--24.

\end{thebibliography}


\end{document}